\documentclass[12pt]{amsart}
\usepackage{amscd, amssymb,latexsym,amsmath, amscd, amsmath}
\usepackage[all]{xy}
\usepackage{anysize}
\usepackage[sc]{mathpazo} 
\usepackage{color}
\marginsize{2.5cm}{2.5cm}{2.5cm}{2.5cm}
\newtheorem{theorem}{Theorem}[section] 
\newtheorem{lemma}[theorem]{Lemma}
\newtheorem{corollary}[theorem]{Corollary}
\newtheorem{proposition}[theorem]{Proposition}
\theoremstyle{definition}

\theoremstyle{remark}
\newtheorem{remark}{Remark}

\begin{document}

\title[Orthogonality of invariant vectors]{Orthogonality of invariant vectors}
\author{U. K. Anandavardhanan and Arindam Jana}

\address{Department of Mathematics, Indian Institute of Technology Bombay, Mumbai - 400076, India.}
\email{anand@math.iitb.ac.in}

\address{School of Mathematics, Harish-Chandra Research Institute, Allahabad - 211019, India.}
\email{arindamjana076@gmail.com}


\date{}

\begin{abstract}
Let $G$ be a finite group with given subgroups $H$ and $K$. Let $\pi$ be an irreducible complex representation of $G$ such that its space of $H$-invariant vectors as well as the space of $K$-invariant vectors are both one dimensional. Let $v_H$ (resp. $v_K$) denote an $H$-invariant (resp. $K$-invariant) vector of unit norm in the standard $G$-invariant inner product $\langle ~,~ \rangle_\pi$ on $\pi$. Our interest is in computing the square of the absolute value of $\langle v_H,v_K \rangle_\pi$. This is the correlation constant $c(\pi;H,K)$ defined by Gross \cite{gro91}. In this paper, we give a sufficient condition for $\langle v_H, v_K \rangle_\pi$ to be zero and a sufficient condition for it to be non-zero (i.e., $H$ and $K$ are correlated with respect to $\pi$), when $G={\rm GL}_2(\mathbb F_q)$, where $\mathbb F_q$ is the finite field of $q=p^f$ elements of odd characteristic $p$, $H$ is its split torus and $K$ is a non-split torus. The key idea in our proof is to analyse the mod $p$ reduction of $\pi$. We give an explicit formula for $|\langle v_H,v_K \rangle_\pi|^2$ modulo $p$. Finally, we study the behaviour of $\langle v_H,v_K \rangle_\pi$ under the Shintani base change and give a sufficient condition for $\langle v_H,v_K \rangle_\pi$ to vanish for an irreducible representation $\pi={\rm BC}(\tau)$ of ${\rm PGL}_2(\mathbb E)$, in terms of the epsilon factor of the base changing representation $\tau$ of ${\rm PGL}_2(\mathbb F)$, where $\mathbb E/\mathbb F$ is a finite extension of finite fields. This is reminiscent of the vanishing of $L(1/2, {\rm BC}(\tau))$, in the theory of automorphic forms, when the global root number of $\tau$ is $-1$.
\end{abstract}

\maketitle

\section{Introduction}\label{introduction}\label{intro}

The study of toric periods, initiated by the works of Waldspurger, occupies a coveted place in number theory and automorphic forms. In these works of Waldspurger, and the later works of many authors, the central object of study is the period integral of a cuspidal representation of $D^\times(\mathbb A_k)$, where $D^\times$ is the invertible elements of a quaternion algebra $D$ over a number field $k$ and $\mathbb A_k$ denotes the ring of ad\`eles of $k$, with respect to the torus defined by a quadratic algebra $K/k$. The toric periods appear naturally in the context of several important themes in number theory such as the study of $L$-values and Shimura correspondence \cite{wal85,wal91}. There are also local analogues of these works \cite{tun83,sai93,pra96}. Toric periods are investigated for the group ${\rm SL}_1(D)$ in \cite{ap13}, where it is highlighted how certain questions in this setting are intimately connected to questions in analytic number theory.   

The present paper does not deal directly with the study of toric periods, rather it is an exercise involving a finite field analogue, whose significance for the test vector problem for toric periods is pointed out by Vatsal \cite{vat17}. We have $G = {\rm GL}_2(\mathbb F_q)$, where $\mathbb F_q$ is the finite field with $q=p^f$ elements, with its split torus $H$ and a non-split torus $K$. Throughout the paper, $p$ is odd. For an irreducible (complex) representation $\pi$ of $G$ that admits a one dimensional space of fixed vectors for both $H$ and $K$, our interest is in evaluating $|\langle v_H,v_K \rangle_\pi|$, where $v_H$ and $v_K$ are $H$ and $K$-invariant vectors of unit norm in the standard $G$-invariant inner product $\langle ~, ~ \rangle$. The main results of this paper give two sufficient conditions for it to vanish (see. Theorem \ref{thm1} (1) and Theorem \ref{thm2}) and one sufficient condition for it not to vanish (see Theorem \ref{thm1} (2)). We also explicitly compute the value of $|\langle v_H,v_K \rangle_\pi|^2$, but this we can do only modulo $p$ (see Theorem \ref{thm1} (2)).     

As has been observed also by Vatsal in \cite{vat17}, though the question deals with complex representations, the answer and the methods, at least partly, lie in the realm of mod $p$ representation theory. Indeed, our first main theorem (cf. Theorem \ref{thm1} (2)) is carried out by investigating the semi-simplifcation of the mod $p$ reduction of a complex representation of ${\rm GL}_2(\mathbb F_q)$, making use of the results of Diamond \cite{dia07}. We stress here that the basic idea of using the Jordan-H\"older factors of the mod $p$ reduction to study the lines which are invariant under $H$ and $K$ is already there in Vatsal's work \cite{vat17}.  

More generally, for any finite group $G$, with given Gelfand subgroups $H$ and $K$ (i.e., $H$ and $K$ are such that the space of $H$-invariant vectors as well as the space of $K$-invariant vectors are at most one dimensional for any irreducible representation of $G$), and an irreducible representation $\pi$ of $G$, let $v_H$ (resp. $v_K$) denote an $H$-invariant (resp. $K$-invariant) vector of unit norm in the standard $G$-invariant inner product $\langle ~,~ \rangle_\pi$ on $\pi$. Then $|\langle v_H,v_K \rangle_\pi|^2$ is precisely the correlation constant $c(\pi;H,K)$ defined by Gross \cite[\S 8]{gro91}. The subgroups $H$ and $K$ are said to be  correlated with respect to $\pi$ if $c(\pi;H,K) \neq 0$. The paper \cite{gro91} is written for $p$-adic and adelic groups and what we attempt in this paper is a finite field analogue in the specific situation of ${\rm GL}(2)$ with its split and non-split tori as the Gelfand subgroups. We may also mention in passing that the correlation constant in a specific set up, where $G=SO(3)\times SO(3)\times SO(3), H=SO(2)\times SO(2)\times SO(2)$ and $K = SO(3)$ embedded diagonally in $G$, comes up naturally in the physics literature in terms of the Racah W-coefficients (see \cite[\S 9]{gro91}).

A final remark before stating the results of this paper more formally is that though our proof makes use only of elementary (complex and modular) representation theory of  ${\rm GL}_2(\mathbb F_q)$, a speculative heuristic which has guided this work was to think of $|\langle v_H,v_K \rangle_\pi|^2$ as an ``$L$-value". We make a few further remarks regarding this heuristic in Section \ref{l-values}.      

Note that we may work with $G={\rm PGL}_2(\mathbb F_q)$ instead of ${\rm GL}_2(\mathbb F_q)$. Fix a character $\chi$ of $\mathbb F_q^\times$ that reduces mod $p$ to the identity character $x \mapsto x$ of $\mathbb F_q^\times$. Similarly, fix a character $\psi$ of $\mathbb F_{q^2}^\times$ that reduces mod $p$ to the identity character of $\mathbb F_{q^2}^\times$. Note that the above steps involve fixing an identification $\overline{\mathbb Q}_p \simeq \mathbb C$ or what amounts, for our purposes, to fixing a prime $\mathfrak p$ of $\overline{\mathbb Q} $ above $p$ (cf. Section \ref{reduction}). Note that the unique quadratic character of ${\rm PGL}_2(\mathbb F_q)$ is given by $\eta = \chi^{\frac{q-1}{2}} \circ \det$. An irreducible principal series representation, up to isomorphism, of ${\rm PGL}_2(\mathbb F_q)$ is of the form ${\rm Ps}(\chi^r,\chi^{-r})$, where $1 \leq r \leq \frac{q-3}{2}$. Irreducible cuspidal representations of ${\rm PGL}_2(\mathbb F_q)$, up to isomorphism, are parametrized by $\psi^{(q-1)r}$, where $1 \leq r \leq \frac{q-1}{2}$. These $(q-2)$  representations and the twisted Steinberg representation St$\otimes \eta$ (and of course the trivial representation) are precisely the representations that admit both $H$ and $K$-invariant vectors. Moreover, for these representations, we have multiplicity one for both the invariant spaces. 

We define the sign, which is nothing but the finite field analogue of the epsilon factor in automorphic representation theory (cf. Remark \ref{rmk-epsilon}), of these representations as follows:
\[
\varepsilon_\pi = 
\begin{cases}
(-1)^r &\text{if $\pi={\rm Ps}(\chi^r,\chi^{-r})$,} \\
(-1)^{\frac{q-1}{2}} &\text{if $\pi={\rm St}\otimes \eta$,} \\
(-1)^{r-1} &\text{if $\pi = \pi(\psi^{(q-1)r})$.}
\end{cases}
\]
Also, write
\[r= \begin{cases}
r_0+r_1p+\dots+r_{f-1}p^{f-1} &\text{if $\pi={\rm Ps}(\chi^r,\chi^{-r})$,} \\
\frac{q-1}{2} &\text{if $\pi = {\rm St}\otimes \eta$,} \\
1+r_0+r_1p+\dots+r_{f-1}p^{f-1} &\text{if $\pi=\pi(\psi^{(q-1)r})$.}
\end{cases}\] 

\begin{remark}\label{rmk-vatsal}
Observe that $r$ and hence the $r_i$'s depend on the choice of a prime ideal $\mathfrak p$ in $\mathbb Z[\zeta_k]$ lying above $p$, where $\zeta_k$ is a primitive $k$-th root of unity, for $k=q-1$ (resp.  $k=q^2-1$) when $\pi$ is a principal series (resp. cuspidal) representation of ${\rm PGL}_2(\mathbb F_q)$. We write
\[r = r(\mathfrak p), r_i = r_i(\mathfrak p),\]
for $\mathfrak p | p$, for the purposes of statement of Theorem \ref{thm1} below.
\end{remark}

Let $\langle ~, ~ \rangle_\pi: \pi \times \pi \rightarrow \mathbb C$ be the standard $G$-invariant inner product on the space of $\pi$. Choose $H$ and $K$-invariant vectors $v_H$ and $v_K$ with $\|v_H\| = \|v_K\| =1$. Let $k$ be as in Remark \ref{rmk-vatsal}. It is easy to see that $|\langle v_H,v_K \rangle_\pi|^2 \in \mathbb Q(\zeta_k) \cap [0,1]$; for instance, this follows from Lemma \ref{lemma-elementary}. For a prime ideal $\mathfrak p$ in $\mathbb Z[\zeta_k]$ lying above $p$, considering $|\langle v_H,v_K \rangle_\pi|^2$ as an element of the completion $\mathbb Q(\zeta_k)_\mathfrak p$ of $\mathbb Q(\zeta_k)$ at $\mathfrak p$, Theorem \ref{thm1} (2) asserts that it is in fact in $\mathbb Z_p$; the value in $\mathbb Z_p$ of course depends on the choice of $\mathfrak p | p$. Theorem \ref{thm1} (2) computes $|\langle v_H,v_K \rangle_\pi|^2$ modulo $p$.   

\begin{theorem}\label{thm1}
Let $\pi$ be an irreducible complex representation of ${\rm PGL}_2(\mathbb F_q)$ as above. Then,
\begin{enumerate}
\item If $\varepsilon_\pi = -1$, then $\langle v_H,v_K \rangle_\pi = 0$.
\item Modulo $p$, for any prime $\mathfrak p$ of $\overline{\mathbb Q}$ lying above $p$, we have,
\[|\langle v_H,v_K \rangle_\pi|^2 =  \begin{cases}
(-1)^{\frac{q-1}{2}} {r(\mathfrak p) \choose r(\mathfrak p)/2} {q-1-r(\mathfrak p) \choose (q-1-r(\mathfrak p))/2} &\text{if each $r_i(\mathfrak p)$ is even,} \\
0 &\text{otherwise.}
\end{cases}\]
In particular, $\langle v_H,v_K \rangle_\pi \neq 0$ if each $r_i(\mathfrak p)$ is even for some $\mathfrak p$ of $\overline{\mathbb Q}$ lying above $p$.
\end{enumerate}    
\end{theorem}

\begin{remark}
Note that we are indeed going modulo $p$, and not just modulo a prime lying above $p$, in the statement of Theorem \ref{thm1} (see also Remark \ref{rmk-modp}).
\end{remark}

\begin{remark}
By Lucas's binomial congruence theorem, when each $r_i$ is even,
\[(-1)^{\frac{q-1}{2}} {r \choose r/2} {q-1-r \choose (q-1-r)/2} \equiv \prod_{i=0}^{f-1} (-1)^{\frac{p-1}{2}} {r_i \choose r_i/2} {p-1-r_i \choose (p-1-r_i)/2} \mod p.\]
\end{remark}

\begin{remark}
We illustrate how the choice of $\mathfrak p | p$ is useful in concluding the non-vanishing of $|\langle v_H,v_K \rangle_\pi|^2$ via a simple example (cf. Example \ref{example2}). 
\end{remark}

It follows from Theorem \ref{thm1} that for an irreducible representation $\pi$ of ${\rm PGL}_2(\mathbb F_p)$, we have a necessary and sufficient condition for orthogonality of the $H$ and $K$-invariant vectors, which we state as a corollary. A version of Corollary \ref{corollary} for the cuspidal case is there in \cite{vat17}.

\begin{corollary}\label{corollary}
When $q=p$, $\langle v_H,v_K \rangle_\pi = 0$ if and only if $\varepsilon_\pi = -1$.
\end{corollary}

\begin{remark}
We expect the statement in Corollary \ref{corollary} to hold true more generally for $q = p^f$ with $f$ odd. However, when $f$ is even, using Theorem \ref{thm2}, we can produce many examples of $\pi$ such that $\langle v_H,v_K \rangle_\pi = 0$ though $\varepsilon_\pi = 1$. 
\end{remark}

Part (1) of Theorem \ref{thm1} is relatively easy and we prove this in Sections \ref{elementary} and \ref{pgl2}. The main ideas in the proof of Part (2) of Theorem \ref{thm1} are as follows. Via an elementary lemma (cf. Lemma \ref{lemma-elementary}) we first convert the problem of determining $|\langle v_H,v_K \rangle_\pi|^2$ to a question about the character of $\pi$. This formulation in terms of character theory enables us to translate the problem to one involving Brauer characters of mod $p$ representations. Together with the mod $p$ analogue of Lemma \ref{lemma-elementary}, which is Lemma \ref{lemma-brauer}, the question of determining $|\langle v_H,v_K \rangle_\pi|^2$ modulo $p$ can be treated at the level of Jordan-H\"older factors of the reduction mod $p$ of the complex representation $\pi$. In doing reduction modulo $p$, we make use of the work of Diamond \cite{dia07} (cf. Section \ref{reduction}). A priori we know, via Brauer character theory, that precisely one Jordan-H\"older factor has a non-zero $H$-invariant vector and precisely one Jordan-H\"older factor has a non-zero $K$-invariant vector. By making use of the results of Diamond, we conclude that these two Jordan-H\"older factors are in fact the same (see Remark \ref{rmk-diamond}). On this Jordan-H\"older factor, that affords both $H$ and $K$-invariant vectors, we can fairly easily carry out the necessary calculations, as the situation is very algebraic (cf. Section \ref{section-proof}).  

Now we state the second main result of this paper. Its proof uses no modular representation theory unlike that of Theorem \ref{thm1}. We show that if $\pi$ is an irreducible principal series representation of ${\rm PGL}_2(\mathbb F_{q^f})$ which is obtained as the base change lift of a representation $\tau$ of ${\rm PGL}_2(\mathbb F)$, then $\langle v_H,v_K \rangle_\pi = 0$ when $\varepsilon_\tau=-1$. We state it as follows.

\begin{theorem}\label{thm2}
Let $\pi={\rm Ps}(\chi,\chi^{-1})$ be a principal series representation of ${\rm PGL}_2(\mathbb F_{q^f})$ which is the base change of an irreducible representation $\tau$ of ${\rm PGL}_2(\mathbb F_q)$. Then, we have
\[\langle v_H, v_K \rangle_\pi = \varepsilon_\tau \langle v_H, v_K \rangle_\pi.\]
In particular, $v_H$ and $v_K$ are orthogonal if $\varepsilon_\tau=-1$.
\end{theorem}

We discuss the theory of Shintani base change \cite{shi76,kaw77} and prove Theorem \ref{thm2} in Section \ref{bc}. The proof in the case when $\tau$ is a principal series representation is somewhat easy whereas the case where $\tau$ is a cuspidal representation is a little bit involved. We rely on the Whittaker model, the Bessel function, and well-known properties of the Gauss sum, to achieve this, in addition to properties of the Shintani base change.
 
We end the paper with a few purely speculative remarks in Section \ref{l-values}, bringing out a possible analogy between $|\langle v_H,v_K \rangle_\pi|^2$ and ``$L$-values" in the theory of automorphic forms.

\section{Elementary lemmas}\label{elementary}

Let $G$ be any finite group. Let $H$ and $K$ be two subgroups of $G$. Let $\pi$ be an irreducible complex representation of $G$. Assume that 
\[\dim \pi^H = 1, \dim \pi^K =1. \] Let 
\[\langle ~, ~ \rangle_\pi: \pi \times \pi \rightarrow \mathbb C\]
be the standard $G$-invariant inner product on the space of $\pi$. Let
\[v_H \in \pi^H, v_K \in \pi^K\] be such that $\|v_H\| = \|v_K\| =1$. Let $X$ (resp. $Y$) denote the operator on $\pi$ given by averaging over $H$ (resp. $K$). Thus,
\[X = \frac{1}{|H|} \sum_{h \in H} \pi(h), Y = \frac{1}{|K|} \sum_{k \in K}\pi(k).\]
We obviously have
\[Xv_H = v_H, Yv_K = v_K.\] Let $s, t \in \mathbb C$ be such that
\[Xv_K = s v_H, Y v_H = t v_K.\]
Now,
\[\langle v_H, v_K \rangle_\pi = \langle X v_H, v_K \rangle_\pi = \langle v_H, X v_K \rangle_\pi = \langle v_H, s v_H \rangle_\pi = \bar{s}\] 
on the one hand, and on the other hand we have,
\[\langle v_H, v_K \rangle_\pi = \langle v_H, Y v_K \rangle_\pi = \langle Y v_H, v_K \rangle_\pi = \langle t v_K, v_K \rangle_\pi = t.\] 
Now we view the operators $X$ and $Y$ in the ordered basis $\{v_H, v_K, \dots\}$. Thus,
\[X= \left(\begin{array}{ccccc} 1 & s & * & \dots & * \\
0 & 0 & 0 & 0 & 0 \\
\vdots & \vdots & \vdots & \vdots & \vdots  \\
0 & 0 & 0 & 0 & 0 
 \end{array} \right), 
Y = \left(\begin{array}{ccccc} 0 & 0 & 0 & 0 & 0 \\
t & 1 & * & \dots & * \\
\vdots & \vdots & \vdots & \vdots & \vdots  \\
0 & 0 & 0 & 0 & 0 
 \end{array} \right).
 \]
 It follows that
 \[{\rm Trace}~(XY) = st = |s|^2 = |\langle v_H, v_K \rangle_\pi|^2.\]
 But,
 \begin{align*}
 {\rm Trace}~(XY) & = {\rm Trace}\left(\frac{1}{|H|}\frac{1}{|K|} \sum_{h \in H} \sum_{k \in K} \pi(h)\pi(k) \right) \\
 & =  {\rm Trace}\left(\frac{1}{|H|}\frac{1}{|K|}  \sum_{h \in H} \sum_{k \in K} \pi(hk) \right) \\
 & = \frac{1}{|H|}\frac{1}{|K|}  \sum_{h \in H} \sum_{k \in K} \chi_\pi(hk), 
 \end{align*}
 where $\chi_\pi$ is the character of $\pi$. 
 
 We have thus proved the following lemma.
 \begin{lemma}\label{lemma-elementary}
 For a representation $\pi$ of $G$ with one dimensional space of invariant vectors for $H$ and $K$, we have the identity:
 \[|\langle v_H, v_K \rangle_\pi|^2 = \frac{1}{|H|}\frac{1}{|K|}  \sum_{h \in H} \sum_{k \in K} \chi_\pi(hk),\]
 where $\chi_\pi$ denotes the character of $\pi$.
 \end{lemma}
It is a difficult question to determine the value of the character sum on the right hand side of the identity in Lemma \ref{lemma-elementary}. However, by the orthogonality relations for irreducible characters, we get,
\begin{equation}\label{eqn-regular}
\sum_{\pi \in \widehat{G}} \dim \pi \cdot |\langle v_H, v_K \rangle_\pi|^2 = \frac{|G||H \cap K|}{|H||K|},
\end{equation}
where $\widehat{G}$ denotes the set of irreducible representations of $G$.

\begin{remark}\label{rmk-non-zero}
It follows from (\ref{eqn-regular}) that, if both $(G,H)$ and $(G,K)$ are multiplicity free pairs, then there are irreducible representations of $G$ for which $\langle v_H, v_K \rangle_\pi$ is non-zero. 
\end{remark}

\begin{remark}
In this remark, we apply Lemma \ref{lemma-elementary} to give a character theoretic proof, for finite groups, of the expression for the correlation constant 
\[c(\pi \otimes \pi^\vee; G, H \times H) = \frac{1}{\dim \pi}\]
given in \cite[p. 291]{gro91}. Here, $G$ is any finite group and $H$ is a Gelfand subgroup, and for an irreducible representation $\pi$ of $G$, we consider the irreducible representation $\pi \otimes \pi^\vee$ of $G \times G$. Indeed, by Lemma \ref{lemma-elementary},
\begin{align*}
c(\pi \otimes \pi^\vee; G, H \times H) &= \frac{1}{|H|^2}\frac{1}{|G|}\sum_{h_1 \in H}\sum_{h_2 \in H}\sum_{g \in G} \chi_{\pi \otimes \pi^\vee}((h_1,h_2)(g,g)) \\
&=  \frac{1}{|H|^2}\frac{1}{|G|}\sum_{h_1 \in H}\sum_{h_2 \in H}\sum_{g \in G} \chi_\pi(h_1g) \overline{\chi_\pi(h_2g)} \\
&= \frac{1}{|H|^2}\frac{1}{|G|}\sum_{h_1 \in H}\sum_{h_2 \in H}\sum_{g \in G} \chi_\pi(h_1g) \chi_\pi(g^{-1}h_2^{-1}) \\
&= 1/\dim \pi,
\end{align*}
by orthogonality relations, together with \cite[Lemma 5.1]{gel70}, according to which
\[\sum_{g \in G} \chi_\pi(h_1g) \chi_\pi(g^{-1}h_2^{-1}) = \frac{1}{\dim \pi} \chi_\pi(h_1h_2^{-1}).
\]
\end{remark}

Now we observe a sufficient condition for $\langle v_H, v_K \rangle_\pi$ to be zero. To this end, assume further that either $H$ has an element of order $2$ that normalizes $K$, say $h_0$, or $K$ has an element of order $2$ that normalizes $H$, say $k_0$. Thus,
\[\pi(k_0)v_H \in \pi^H, \pi(h_0)v_K \in \pi^K,\]
and moreover,
\[\pi(k_0)v_H = \mu_\pi(k_0)v_H, \pi(h_0)v_K = \nu_\pi(h_0)v_K,\]
where $\mu_\pi(k_0) = \pm 1$ and $\nu_\pi(h_0) = \pm 1$.

\begin{lemma}\label{lemma-sign}
Under all of the assumptions above, if $\mu_\pi(k_0)$ or $\nu_\pi(h_0)$ is $-1$ then $v_H$ and $v_K$ are orthogonal.  
\end{lemma}

\begin{proof}
We take the case of $k_0$ and the other case is similar. If $\mu_\pi(k_0)=-1$, then observe that
\[\langle v_H, v_K \rangle_\pi = \langle v_H, \pi(k_0) v_K \rangle_\pi = \langle \pi(k_0)v_H, v_K \rangle_\pi = \mu_\pi(k_0)\langle v_H, v_K \rangle_\pi = - \langle v_H, v_K \rangle_\pi, \]
and we see that $\langle v_H, v_K \rangle_\pi = 0$.
\end{proof}

\section{$G={\rm PGL}_2(\mathbb F_q)$}\label{pgl2}

Let $\mathbb F_q$ be the finite field of $q$ elements where $q=p^f$ for an odd prime number $p$. Let $G^\prime ={\rm GL}_2(\mathbb F_q)$. Let $H^\prime \simeq \mathbb F_q^\times \times \mathbb F_q^\times$ be the split torus and let $K^\prime \simeq \mathbb F_{q^2}^\times$ be the non-split torus of $G^\prime$. Since we are dealing with irreducible representations of $G^\prime$ with trivial central character, we may as well work with $G={\rm PGL}_2(\mathbb F_q)$. Let $H$ (resp. $K$) denote the image of $H^\prime$ (resp. $K^\prime$) under the quotient map. We may take 
\[H = \left\{\left(\begin{array}{cc} a & 0 \\ 0 & 1 \end{array} \right) \mid a \in \mathbb F_q^\times\right\},\]
and we fix $\alpha \in \mathbb F_q^\times$ such that 
\[k_\alpha= \left(\begin{array}{cc} 1 & \alpha \\ 1 & 1 \end{array} \right) \] has order $q+1$ in $G$, and consider the cyclic group $K$ generated by $k_\alpha \in G$. Thus,  
\[K = \left\{ \left(\begin{array}{cc} 1 & \alpha z \\ z & 1 \end{array} \right), \left(\begin{array}{cc} 0 & \alpha \\ 1 & 0 \end{array} \right)  \mid z \in \mathbb F_q  \right\}.\]
\begin{remark}\label{rmk-k}
The group $K = K_\alpha$ depends on the choice of $\alpha$, however if we choose another $\beta$ such that $k_\beta$ generates a group $K_\beta$ of order $q+1$, then it can be seen that $K_\beta$ and $K_\alpha$ are conjugates by an element of $H$. Thus, $|\langle v_H, v_{K_\alpha} \rangle_\pi|$ is independent of the choice of $\alpha$, and for this reason we work with a fixed $\alpha$ and suppress the dependence of $K$ on the choice of $\alpha$ for ease of writing. 
\end{remark}

In the notations of Section \ref{elementary}, we have
\[h_0 = \left(\begin{array}{rc} -1 & 0 \\ 0 & 1 \end{array} \right), k_0 = \left(\begin{array}{cc} 0 & \alpha \\ 1 & 0 \end{array} \right).\]
Note that we also have
\[k_0hk_0^{-1}=h^{-1}, h_0kh_0^{-1}=k^{-1}\]
for all $h \in H, k \in K$. 

Our first observation in this section is that, for $G={\rm PGL}_2(\mathbb F_q)$, there is no ambiguity between the two signs $\mu_\pi(k_0)$ and $\nu_\pi(h_0)$ introduced in Section \ref{elementary}; they coincide. In the next proposition, we show this and also determine this common sign, which we denote by $\varepsilon_\pi$. This notation is suggestive as it is, in effect, the usual $\epsilon$-factor as we are going to see below from Proposition \ref{prop-agree} (cf. Remark \ref{rmk-epsilon}).

Before stating the proposition, we recall that $G = {\rm PGL}_2(\mathbb F_q)$ has $(q+2)$ many isomorphism classes of irreducible representations given by:
\begin{enumerate}
\item $2$ one dimensional representations given by the trivial representation and a character $\eta$ of order $2$,
\item the Steinberg representation, denoted by St, and its twist ${\rm St}\otimes \eta$, of dimension $q$ each,
\item $\frac{q-3}{2}$ many principal series representations, of dimension $q+1$ each, denoted by ${\rm Ps}(\chi,\chi^{-1})$, where $\chi$ is a character of $\mathbb F_q^\times$ with $\chi^2\neq 1$, 
\item $\frac{q-1}{2}$ many cuspidal representations, of dimension $q-1$ each, denoted by $\pi(\psi)$, associated to a character $\psi$ of $\mathbb F_{q^2}^\times$ such that $\psi\neq \psi^q = \psi^{-1}$.   
\end{enumerate}

Out of these $(q+2)$ many irreducible representations, the character $\eta$ has no $H$-fixed or $K$-fixed vector, and the Steinberg representation has no $K$-fixed vector but two linearly independent $H$-fixed vectors. For the trivial representation, obviously $|\langle v_H, v_K \rangle_\pi|$ is $1$. Thus, the analysis of $|\langle v_H, v_K \rangle_\pi|$ needs to be carried out only for the remaining $(q-1)$ irreducible representations. 

\begin{remark}\label{rmk-regular}
Note that (\ref{eqn-regular}) reads as
\[\sum_{\pi \in \widehat{G}} \dim \pi ~|\langle v_H, v_K \rangle_\pi|^2 = q.\]
\end{remark}
 
\begin{proposition}\label{prop-agree}
Let $\pi$ be an irreducible representation of ${\rm PGL}_2(\mathbb F_q)$. Then,
\[\varepsilon_\pi = \mu_\pi(k_0)=\nu_\pi(h_0)= \begin{cases}
\chi(-1) &\text{if $\pi={\rm Ps}(\chi,\chi^{-1})$ or ${\rm St} \otimes \chi$,} \\
-\psi(\sqrt{\delta}) &\text{if $\pi=\pi(\psi)$,}
\end{cases}\]
where $\delta$ is any non-square in $\mathbb F_q^\times$.
\end{proposition}

\begin{proof}
As a sanity check, note that $\chi(-1)$ is obviously $\pm 1$ and $-\psi(\sqrt{\delta})$ is also $\pm 1$ since $\psi$ is given to be trivial on $\mathbb F_q^\times$. Now to prove the proposition, we first argue as in the proof of Lemma \ref{lemma-elementary}, taking $X$ to be $H$-average operator and taking $\pi(k_0)$ instead of $Y$ in that proof. The calculations there then give us
\[\mu_\pi(k_0) = \frac{1}{|H|} \sum_{h \in H} \chi_\pi(hk_0). \] Indeed, the matrix $X \pi(k_0)$ will have $\mu_\pi(k_0)$ as the first entry in the first row and will have zero rows from the second row onwards. Similarly, taking $Y$ to be $K$-average operator and taking $\pi(h_0)$ instead $X$, we get
\[\nu_\pi(h_0) = \frac{1}{|K|} \sum_{k \in K} \chi_\pi(h_0k). \]
Now both $hk_0$ and $h_0k$ are elements of order $2$. There are two conjugacy classes of order $2$ in $G$ represented by $h_0$ and $k_0$, say $[h_0]$ and $[k_0]$ respectively. As $h$ varies in $H$, the product $hk_0$ belongs to these two conjugacy classes the same number of times; i.e., half each. Similar is the case of $h_0k$ when $k$ varies in $K$. Now the assertion follows from the character table of ${\rm PGL}_2(\mathbb F_q)$. Indeed, for $\pi={\rm Ps}(\chi,\chi^{-1})$, we have $\chi_\pi([h_0])=2\chi(-1)$ and $\chi_\pi([k_0])=0$, for $\pi=\pi(\psi)$, we have $\chi_\pi([h_0])=0$ and $\chi_\pi([k_0])=-2\psi(\sqrt{\delta})$, and for $\pi={\rm St}\otimes \eta$, we have $\chi_\pi([h_0])=\eta(-1)$ and $\chi_\pi([k_0])=-\eta(-\delta)=\eta(-1)$.
\end{proof}

\begin{remark}
Note that in the proof of Proposition \ref{prop-agree}, we need not have a priori observed that the products $hk_0$ for $h \in H$ (resp. $h_0k$ for $k \in K$) belong to the conjugacy classes $[h_0]$ and $[k_0]$ an equal number of times. This is automatic as we already know that the object that we compute is $\pm 1$.
\end{remark}

\begin{remark}
Both the character sum expressions in the proof of Proposition \ref{prop-agree} are $0$ for $\eta$ and the Steinberg representation. It is obviously $1$ for the trivial representation. These signs are distributed amongst the rest of the $(q-1)$ many representations so that
\[\sum_{\pi \in \widehat{G}} \dim \pi \cdot \varepsilon_\pi = 0,\]
as can be seen, for instance, by looking at the decomposition of the character of the regular representation.   
\end{remark}

\begin{remark}\label{rmk-epsilon}
Note that the sign $\varepsilon_\pi$ in Proposition \ref{prop-agree} is related to the $\epsilon$-factor as follows. If $\pi={\rm Ps}(\chi,\chi^{-1})$, then
\[\epsilon(\pi)=\epsilon(\chi)\epsilon(\chi^{-1})=\chi(-1).\] If $\pi=\pi(\psi)$, then
\[\epsilon(\pi)=\epsilon(\psi)=\psi(\sqrt{\delta}).\] Thus, the sign $\varepsilon_\pi$ equals the $\epsilon$-factor for a principal series representation, whereas for a cuspidal representation it is the negative of the $\epsilon$-factor.
\end{remark}

\begin{remark}
In Section \ref{l-values}, we will relate the role that the sign $\varepsilon_\pi$ in Proposition \ref{prop-agree} plays in the vanishing of $\langle v_H, v_K \rangle_\pi$ with the analogous role that the global root number plays in the vanishing of an automorphic $L$-function. A negative sign is sufficient for vanishing, however it does not capture all the vanishing, and moreover the vanishing can be for subtle reasons!
\end{remark}

Now fix a character $\chi$ that generates the character group of $\mathbb F_q^\times$. Note that $\eta = \chi^{\frac{q-1}{2}} \circ \det$. An irreducible principal series representation, up to isomorphism, of ${\rm PGL}_2(\mathbb F_q)$ is of the form ${\rm Ps}(\chi^r,\chi^{-r})$, where $1 \leq r \leq \frac{q-3}{2}$. Similarly, fix a character $\psi$ that generates the character group of $\mathbb F_{q^2}^\times$. Irreducible cuspidal representations of ${\rm PGL}_2(\mathbb F_q)$, up to isomorphism, are parametrized by $\psi^{(q-1)r}$, where $1 \leq r \leq \frac{q-1}{2}$. 
\begin{corollary}\label{cor-sign}
In the notations above,
\[
\varepsilon_\pi = 
\begin{cases}
(-1)^r &\text{if $\pi={\rm Ps}(\chi^r,\chi^{-r})$,} \\
(-1)^{\frac{q-1}{2}} &\text{if $\pi={\rm St}\otimes \eta$,} \\
(-1)^{r-1} &\text{if $\pi = \pi(\psi^{(q-1)r})$.}
\end{cases}
\]

\end{corollary}

\begin{remark}\label{rmk-sign}
Note that the sign does not depend on the choice of the generator of the relevant character groups.
\end{remark}

\section{Reduction mod $p$}\label{reduction}

The Jordan-H\"older constituents of the reduction mod $p$ of complex irreducible representations of ${\rm GL}_2(\mathbb F_q)$ are computed in \cite{dia07}. In this section we recall the results due to Diamond \cite[Proposition 1.1 ~\&~ Proposition 1.3]{dia07} which play a crucial role in the proof of Theorem \ref{thm1}. We closely follow the exposition in \cite{dia07}.

Let $\overline{\mathbb F}_p$ denote an algebraic closure of $\mathbb F_p$. The number of $p$-regular conjugacy classes in ${\rm GL}_2(\mathbb F_q)$ is $q(q-1)$ and thus it has that many irreducible representations, up to isomorphism, over $\overline{\mathbb F}_p$. These are given in terms of the symmetric power representations of the two dimensional standard representation of ${\rm GL}_2(\mathbb F_q)$. 

Let $q=p^f$ and let $r \in \mathbb N$ be such that $0 \leq r \leq q$. Write
\[r = r_0+r_1p+\dots+r_{f-1}p^{f-1}\]
with $0 \leq r_i \leq p-1$ for $0 \leq i \leq f-1$. Let ${\rm Sym}^{r_i} \overline{\mathbb F}_p^2$ denote the $r_i$-th symmetric power representation of the standard representation of ${\rm GL}_2(\mathbb F_q)$. Then 
 \[\rho_r = {\rm Sym}^{r_0} \overline{\mathbb F}_p^2 \otimes {\rm Sym}^{r_1} \overline{\mathbb F}_p^2 \circ {\rm Frob} \otimes \dots \otimes  {\rm Sym}^{r_{f-1}} \overline{\mathbb F}_p^2 \circ {\rm Frob}^{f-1}\] 
is an irreducible representation of ${\rm GL}_2(\mathbb F_q)$, where Frob is the Frobenius morphism. There are $q$ many such irreducible representations and together with the $(q-1)$ many character twists we get all the irreducible mod $p$ representations of ${\rm GL}_2(\mathbb F_q)$.

Recall that the Brauer character $\chi_\rho^{\rm Br}$ of a mod $p$ representation $\rho$ of a finite group $G$ is the complex valued class function, with respect to a fixed embedding $i$ of $\overline{\mathbb F}_p^\times$ into $\mathbb C^\times$, given by 
\[\chi_\rho^{\rm Br}(g) = \sum_j i(\lambda_j(g))\] where $\lambda_j(g)$ are the eigenvalues of $\rho(g)$. The semi-simplification of a mod $p$ representation is determined by its Brauer character.

For the reduction mod $p$ of a finite dimensional complex representation of $G$, we first treat is as a $\overline{\mathbb Q}_p$-representation, via an identification $\overline{\mathbb Q}_p \simeq \mathbb C$. The representation admits a $G$-stable $\overline{\mathbb Z}_p$-lattice, and the mod $p$ reduction of the given representation is obtained by reducing this lattice modulo the maximal ideal of $\overline{\mathbb Z}_p$. The character of the given representation and the Brauer character of the mod $p$ reduced representation agree on the $p$-regular conjugacy classes of $G$. It follows that the semi-simplification of the mod $p$ reduced representation is independent of the choice of the lattice.

The following is \cite[Proposition 1.1]{dia07} which gives reduction mod $p$ of principal series representations of ${\rm GL}_2(\mathbb F_q)$.

\begin{proposition}[Diamond]\label{diamond-ps}
Let $\pi = {\rm Ps}(1,\chi)$ be an irreducible principal series representation of ${\rm GL}_2(\mathbb F_q)$ for a character $\chi$ of $\mathbb F_q^\times$. Suppose the mod $p$ reduction of $\chi$ is given by $x \mapsto x^r$, where $r = r_0+r_1p+\dots +r_{f-1}p^{f-1}$ with $0 \leq r_i \leq p-1$. Then the Jordan-H\"older constituents of the semi-simplification of the mod $p$ reduction of $\pi$ consist of the representations
\[ \bigotimes_{i=0}^{f-1} \left( {\rm Sym}^{(n_{J,i}-1)} \overline{\mathbb F}_p^2 \otimes \det{^{m_{J,i}}} \overline{\mathbb F}_p^2 \right) \circ {\rm Frob}^i,\]
where for $J \subset \mathbb Z/f\mathbb Z$,
\[n_{J,i} = \begin{cases}
r_i+\delta_{J+1}(i) &\text{if $i \in J$,} \\
p-r_i-\delta_{J+1}(i) &\text{if $i \notin J$,}
\end{cases}\]
and
\[m_{J,i} = \begin{cases}
0 &\text{if $i \in J$,} \\
r_i+\delta_{J+1}(i) &\text{if $i \notin J$.}
\end{cases}\]
Here, $\delta_S$ denotes the characteristic function of a set $S$. 
\end{proposition}

Fix a character $\chi:\mathbb F_q^\times \rightarrow \mathbb C^\times$ whose mod $p$ reduction is the identity character $x \mapsto x$. Our interest is in the reduction mod $p$ of the principal series representation $\pi={\rm Ps}(\chi^{-r},\chi^r)$ of ${\rm PGL}_2(\mathbb F_q)$. In order to apply Proposition \ref{diamond-ps}, we consider $\pi \otimes \chi^r = {\rm Ps}(1,\chi^{2r})$. Let
\[r=r_0+r_1p+\dots+r_{f-1}p^{f-1}\]
with $0 \leq r_i \leq p-1$ for $0 \leq i \leq f-1$. Note that the possible coefficients in the $p$-adic expansion of $2r$ are
\[2r_i, 2r_i+1,2r_i-p,2r_i-p+1,\]
for $0 \leq i \leq f-1$. Note that if the $i$-th coefficient is $2r_i+1$ or $2r_i-p+1$, then the $(i-1)$-st coefficient is 
$2r_{i-1}-p$ or $2r_{i-1}-p+1$ and if the $i$-th coefficient is $2r_i$ or $2r_i-p$, then the $(i-1)$-st coefficient is 
$2r_{i-1}$ or $2r_{i-1}+1$.

Now let 
\[J = \left\{i \mid \mbox{the $i$-th coefficient in the $p$-adic expansion of $2r$ is $2r_i$ or $2r_i+1$}\right\}.\]
It follows that for $i \in J$, we get the $i$-th component of the mod $p$ reduction of $\pi ={\rm Ps}(\chi^{-r},\chi^r)$ to be 
\[\left( {\rm Sym}^{2r_i} \overline{\mathbb F}_p^2 \otimes \det{^{-r_i}} \overline{\mathbb F}_p^2 \right) \circ {\rm Frob}^i,\] 
and for $i \notin J$, we get
\[\left( {\rm Sym}^{2p-2-2r_i} \overline{\mathbb F}_p^2 \otimes \det{^{r_i}} \overline{\mathbb F}_p^2 \right) \circ {\rm Frob}^i.\] 

We state the above observation as a proposition for future use.

\begin{proposition}\label{prop-future-1}
Consider $\pi={\rm Ps}(\chi^{-r},\chi^r)$. Write $r=r_0+r_1p+\dots+r_{f-1}p^{f-1}$ with $0 \leq r_i \leq p-1$ for $0 \leq i \leq f-1$. Let $J \subset \mathbb Z/f\mathbb Z$ be defined as
\[J = \left\{i \mid \mbox{the $i$-th coefficient in the $p$-adic expansion of $2r$ is $2r_i$ or $2r_i+1$}\right\}.\] 
Then the component in the mod $p$ reduction of $\pi$ corresponding to $J$ is
\[ \left( \bigotimes_{i \in J} \left( {\rm Sym}^{2r_i} \overline{\mathbb F}_p^2 \otimes \det{^{-r_i}} \overline{\mathbb F}_p^2 \right) \circ {\rm Frob}^i \right) \bigotimes \left(  \bigotimes_{i \notin J} \left( {\rm Sym}^{2p-2-2r_i} \overline{\mathbb F}_p^2 \otimes \det{^{r_i}} \overline{\mathbb F}_p^2 \right) \circ {\rm Frob}^i \right).\]
\end{proposition}

The following is \cite[Proposition 1.3]{dia07} which gives reduction mod $p$ of cuspidal representations of ${\rm GL}_2(\mathbb F_q)$.

\begin{proposition}[Diamond]\label{diamond-cusp}
Let $\pi=\pi(\psi)$ be a cuspidal representation of ${\rm GL}_2(\mathbb F_q)$ associated to a character $\psi$ of $\mathbb F_{q^2}^\times$. Suppose the mod $p$ reduction of $\psi$ is given by $x \mapsto x^r$, where without loss of generality we may assume $1 \leq r \leq q$ (by twisting $\pi$ by a character of $\mathbb F_q^\times$ if necessary) and write $r = 1+r_0+r_1p+\dots+r_{f-1}p^{f-1}$  with $0 \leq r_i \leq p-1$. Then the Jordan-H\"older constituents of the semi-simplification of the mod $p$ reduction of $\pi$ consist of the representations
\[ \bigotimes_{i=0}^{f-1} \left( {\rm Sym}^{(n_{J,i}-1)} \overline{\mathbb F}_p^2 \otimes \det{^{m_{J,i}}} \overline{\mathbb F}_p^2 \right) \circ {\rm Frob}^i,\]
where for $J \subset \mathbb Z/f\mathbb Z$,
\[n_{J,i} = \begin{cases}
r_i+1-\delta_{J+1}(i) &\text{if $i=0 \in J$,} \\
p-r_i-1+\delta_{J+1}(i) &\text{if $i =0 \notin J$,} \\
r_i+\delta_{J+1}(i) &\text{if $i \in J, i \neq 0$,} \\
p-r_i-\delta_{J+1}(i) &\text{if $i \notin J, i \neq 0$,} \\
\end{cases}\]
and
\[m_{J,i} = \begin{cases}
\delta_{J+1}(i) &\text{if $i =0  \in J$,} \\
r_i+1 &\text{if $i =0 \notin J$,} \\
0 &\text{if $i \in J, i \neq 0$,} \\
r_i+\delta_{J+1}(i) &\text{if $i \notin J, i \neq 0$.} \\
\end{cases}\]
Here, $\delta_S$ denotes the characteristic function of a set $S$. 
\end{proposition}

Fix a character $\psi:\mathbb F_{q^2}^\times \rightarrow \mathbb C^\times$ whose mod $p$ reduction is the identity character $x \mapsto x$. We are interested in $\pi=\pi(\psi^{(q-1)r})$, where $1 \leq r \leq (q-1)/2$. We write 
\[(1-q)r = 2r-r(q+1).\]
Now Proposition \ref{diamond-cusp} applies to $\pi \otimes \chi^r$, where $\chi$ is the restriction of $\psi$ to $\mathbb F_q^\times$. Write
\[r = 1+ r_0+r_1p+\dots+r_{f-1}p^{f-1}\]
with $0 \leq r_i \leq p-1$ for $0 \leq i \leq f-1$. Note that $r_{f-1} \leq (p-1)/2$. Thus,
\[2r = 1 + a_0+a_1p+\dots+a_{f-1}p^{f-1}\]
where $a_0$ is $1+2r_0$ or $1+2r_0-p$. For $i \neq 0$, $a_i$ is possibly
\[2r_i, 2r_i+1, 2r_i-p, 2r_i-p+1.\] 

A similar analysis as in the case of principal series leads us to the following proposition regarding the mod $p$ reduction of a cuspidal representation of ${\rm PGL}_2(\mathbb F_q)$. Note that Proposition \ref{prop-future-2} is anticipated in \cite[p. 13]{vat17}.

\begin{proposition}\label{prop-future-2}
Consider $\pi=\pi(\psi^{(q-1)r})$. Write $r=1+r_0+r_1p+\dots+r_{f-1}p^{f-1}$ with $0 \leq r_i \leq p-1$ for $0 \leq i \leq f-1$. Let $J \subset \mathbb Z/f\mathbb Z$ be defined as
\[J = \left\{i \mbox{~if $i \neq 0$ and $a_i = 2r_i$ or $2r_i+1$} \right\} \cup \{0 \mbox{~if $a_0=1+2r_0$}\}.\] 
Then the component in the mod $p$ reduction of $\pi$ corresponding to $J$ is
\[ \left( \bigotimes_{i \in J} \left( {\rm Sym}^{2r_i} \overline{\mathbb F}_p^2 \otimes \det{^{-r_i}} \overline{\mathbb F}_p^2 \right) \circ {\rm Frob}^i \right) \bigotimes \left(  \bigotimes_{i \notin J} \left( {\rm Sym}^{2p-2-2r_i} \overline{\mathbb F}_p^2 \otimes \det{^{r_i}} \overline{\mathbb F}_p^2 \right) \circ {\rm Frob}^i \right).\]
\end{proposition}

\section{Proof of Theorem \ref{thm1}}\label{section-proof}

In this section, we compute the value of $|\langle v_H,v_K \rangle_\pi|^2$ modulo $p$ by making use of the mod $p$ reduction theory of complex representations of ${\rm PGL}_2(\mathbb F_q)$ and in particular the results due to Diamond \cite[Proposition 1.1 ~\&~ Proposition 1.3]{dia07} (cf. Section \ref{reduction}). 

Let $\rho$ be an irreducible mod $p$ representation of ${\rm PGL}_2(\mathbb F_q)$ (cf. Section \ref{reduction}). Suppose $\rho$ has fixed vectors for $H$ and $K$ (cf. Section \ref{pgl2}). It is easy to see that they are unique up to multiplication by scalars. This can be done via the Brauer character table or alternatively this can be deduced from the corresponding result for complex representations via mod $p$ reduction. We fix a choice of fixed vectors for $H$ and $K$ and denote them respectively by $v_H$ and $v_K$.    

We start by observing that the elementary argument in Section \ref{elementary} involving 
\[X = \frac{1}{|H|} \sum_{h \in H} \pi(h), Y = \frac{1}{|K|} \sum_{k \in K}\pi(k),\]
with 
\[Xv_H = v_H, Yv_K = v_K, Xv_K = s v_H, Y v_H = t v_K, \]
where $s,t \in \mathbb C$, can be repeated mutatis mutandis to get the following lemma.

\begin{lemma}\label{lemma-brauer}
Let $\rho$ be an irreducible mod $p$ representation of ${\rm PGL}_2(\mathbb F_q)$. Fix an embedding $i: \overline{\mathbb F}_p^\times \hookrightarrow \mathbb C^\times$ with respect to which we have the Brauer character $\chi_\rho^{\rm Br}$ of $\rho$. Let 
\[X = \frac{1}{|H|} \sum_{h \in H} \rho(h), Y = \frac{1}{|K|} \sum_{k \in K}\rho(k),\]
and let $s,t \in \overline{\mathbb F}_p$ be such that
\[Xv_K = s v_H ~\&~ Y v_H = t v_K.\]
Then, modulo $p$, we have
\[st = \frac{1}{|H|}\frac{1}{|K|}  \sum_{h \in H} \sum_{k \in K} \chi_\rho^{\rm Br}(hk).\]
\end{lemma}

\begin{remark}\label{rmk-unipotent}
The group ${\rm PGL}_2(\mathbb F_q)$ has exactly one conjugacy class which is not $p$-regular and this class is unipotent. It is possible that $hk$ is unipotent. Indeed, it can be checked that 
\[\left|\left\{(h,k) \mid hk \mbox{~is unipotent} \right\} \right| = \begin{cases}
q-1 &\text{if $p \equiv 1 \mod 4$,} \\
q-3 &\text{if $p \equiv 3 \mod 4$.}
\end{cases}\]
Note that the Brauer character is defined on a unipotent conjugacy class as well where its value is $\dim \rho$.
\end{remark}

Now we make use of the fact that the character of a complex representation agrees with the Brauer character of its reduction modulo $p$ on $p$-regular conjugacy classes. Let $\pi$ be a complex representation of ${\rm PGL}_2(\mathbb F_q)$. Let JH$(\pi)$ denote the set of the Jordan-H\"older constituents of its reduction modulo $p$. Let $\mathcal U$ denote the unipotent conjugacy class of ${\rm PGL}_2(\mathbb F_q)$. Note that
\begin{align*}
\sum_{h \in H}\sum_{k\in K} \chi_\pi(hk) &= \mathop{\sum\sum}_{hk \notin \mathcal U} \chi_\pi(hk) + \mathop{\sum\sum}_{hk \in \mathcal U} \chi_\pi(hk) \\
&= \sum_{\rho \in {\rm JH}(\pi)} \mathop{\sum\sum}_{hk \notin \mathcal U} \chi_\rho^{\rm Br}(hk) + \mathop{\sum\sum}_{hk \in \mathcal U} \chi_\pi(hk) \\
&= \sum_{\rho \in {\rm JH}(\pi)} \sum_{h \in H} \sum_{k \in K} \chi_\rho^{\rm Br}(hk) + \mathop{\sum\sum}_{hk \in \mathcal U} \chi_\pi(hk) - \sum_{\rho \in {\rm JH}(\pi)}  \mathop{\sum\sum}_{hk \in \mathcal U} \dim \rho \\
&= \sum_{\rho \in {\rm JH}(\pi)} \sum_{h \in H} \sum_{k \in K} \chi_\rho^{\rm Br}(hk) + \mathop{\sum\sum}_{hk \in \mathcal U} \chi_\pi(hk) -  \mathop{\sum\sum}_{hk \in \mathcal U} \dim \pi \\
&= \sum_{\rho \in {\rm JH}(\pi)} \sum_{h \in H} \sum_{k \in K} \chi_\rho^{\rm Br}(hk) - q |\mathcal U|,
\end{align*}
since 
\[
(\chi_\pi(\mathcal U),\dim \pi) = 
\begin{cases}
(1,q+1) &\text{if $\pi$ is a principal series,} \\
(0,q) &\text{if $\pi$ is (twisted) Steinberg,}\\
(-1,q-1) &\text{if $\pi$ is cuspidal.}
\end{cases}
\]

The above arguments, together with Lemma \ref{lemma-elementary}, lead to the following proposition. 

\begin{proposition}\label{prop-modulo}
Modulo $p$, we have
 \[|\langle v_H, v_K \rangle_\pi|^2 = \sum_{\rho \in {\rm JH}(\pi)} s_\rho t_\rho,\]
where
\[X_\rho = \frac{1}{|H|} \sum_{h \in H} \rho(h), Y_\rho = \frac{1}{|K|} \sum_{k \in K}\rho(k),\]
and $s_\rho$ and $t_\rho$ are defined by 
 \[X_\rho v_{K,\rho} = s_\rho v_{H,\rho} ~\&~ Y_\rho  v_{H,\rho} = t_\rho v_{K,\rho}.\]
\end{proposition}

\begin{remark}\label{rmk-diamond}
For an irreducible complex representation of ${\rm PGL}_2(\mathbb F_q)$ with one dimensional spaces of $H$ and $K$-invariant vectors, it follows by considering Brauer characters that there is precisely one irreducible Jordan-H\"older constituent in its reduction mod $p$ with an $H$-fixed vector. Similarly for $K$ as well. This is because Brauer character theory works well as $H$ and $K$ have orders prime to $p$. These two Jordan-H\"older constituents could be different a priori but the key observation is that they are not different. This will follow in the next subsections where we show that the irreducible mod $p$ representation that we have isolated in Proposition \ref{prop-future-1} (resp. Proposition \ref{prop-future-2})  for a principal series representation (resp. for a cuspidal representation) of ${\rm PGL}_2(\mathbb F_q)$ admits both $H$ and $K$-fixed vectors. Indeed, we are going to exhibit them explicitly.   
\end{remark}

\begin{remark}
A consequence of Remark \ref{rmk-diamond} is that all the summands on the right hand side of the identity in Proposition \ref{prop-modulo} are zero except possibly one.    
\end{remark}

\subsection{The case of $q=p$}\label{case-1}

As mentioned in Remark \ref{rmk-diamond}, the key to proving the main result lies in Proposition \ref{prop-future-1} and Proposition \ref{prop-future-2} (together with Proposition \ref{prop-modulo}). Note that when $q=p$, we need to consider 
\[\rho = {\rm Sym}^{2r} \overline{\mathbb F}_p^2 \otimes \det{^{-r}} \overline{\mathbb F}_p^2,\]
for some $0 \leq r \leq p-1$. A basis for this representation is given by $\{x^{2r-i}y^i\}$ for $i = 0, \dots, 2r$. Recall that the action of $\rho$ is given by
\[\rho \left( \left(\begin{array}{cc} a & b \\ c & d \end{array} \right)\right)x^{2r-i}y^i = (ax+cy)^{2r-i}(bx+dy)^i (ad-bc)^{-r}.\]
As the center of ${\rm GL}_2(\mathbb F_q)$ acts trivially on $\rho$, it is indeed a representation of ${\rm PGL}_2(\mathbb F_q)$.

It is easy to exhibit $H$ and $K$-fixed vectors in $\rho$. We may take
\begin{equation}\label{eqn-vh}
v_H = v_{H,\rho} = x^ry^r.
\end{equation}
Indeed,
\[\left(\begin{array}{cc} a & 0 \\ 0 & 1 \end{array} \right) \cdot x^ry^r = (ax)^ry^r \cdot a^{-r} = x^ry^r. \] 

Recall that we fixed $\alpha \in \mathbb F_p^\times$ such that 
\[k_\alpha= \left(\begin{array}{cc} 1 & \alpha \\ 1 & 1 \end{array} \right) \] has order $p+1$ in ${\rm PGL}_2(\mathbb F_p)$, and 
\[K = \left\{ \left(\begin{array}{cc} 1 & \alpha z \\ z & 1 \end{array} \right), \left(\begin{array}{cc} 0 & \alpha \\ 1 & 0 \end{array} \right)  \mid z \in \mathbb F_p  \right\}.\]
Now note that, 
\begin{align*}
\left(\begin{array}{cc} 1 & \alpha \\ 1 & 1 \end{array} \right) \cdot (\alpha x^2-y^2) &= \alpha(x+y)^2-(\alpha x+y)^2  \cdot (1-\alpha)^{-1} \\
&= \left[ \alpha(x^2+2xy+y^2)-(\alpha^2x^2+2\alpha xy+y^2)\right] \cdot (1-\alpha)^{-1} \\
&= \alpha x^2-y^2.
\end{align*}
It follows that
\begin{equation}\label{eqn-vk}
v_K = v_{K,\rho} = (\alpha x^2-y^2)^r.
\end{equation}

Now, as before, let
\[X = X_\rho = \frac{1}{|H|} \sum_{h \in H} \rho(h), Y = Y_\rho = \frac{1}{|K|} \sum_{k \in K}\rho(k).\]
Our interest is in $s=s_\rho$ and $t=t_\rho$ such that 
\[Xv_K = s v_H ~\&~ Y v_H = t v_K.\]

Note that
\[\left(\begin{array}{cc} a & 0 \\ 0 & 1 \end{array} \right) \cdot (\alpha x^2-y^2)^r = (\alpha a x^2-a^{-1}y^2)^r,\]
and therefore
\[Xv_K = \frac{1}{|H|} \sum_{a \in \mathbb F_p^\times} (\alpha a x^2-a^{-1}y^2)^r =
\begin{cases}
(-1)^{r/2} {r \choose r/2}\alpha^{r/2} \cdot x^ry^r &\text{if $r$ is even,} \\
0 &\text{if $r$ is odd,}
\end{cases}
\]
as can be seen by first applying the binomial theorem and then taking the sum.

Next we look for $t$ such that $Yv_H = tv_K$. Note that
\begin{align*}
Y \cdot x^ry^r &= \frac{1}{|K|} \left[ \left(\begin{array}{cc} 0 & \alpha \\ 1 & 0 \end{array}\right) \cdot x^ry^r + \sum_{z=0}^{p-1} \left(\begin{array}{cc} 1 & \alpha z \\ z & 1 \end{array}\right) \cdot x^ry^r   \right] \\
&=  \frac{1}{|K|} \left[ (-1)^r x^ry^r + \sum_{z=0}^{p-1} \frac{(x+yz)^r(\alpha xz+y)^r}{(1-\alpha z^2)^r} \right].
\end{align*}
Equating this expression to 
\[t \cdot (\alpha x^2 - y^2)^r,\]
and looking at the coefficient of $x^{2r}$, we get
\begin{align*}
t &= \sum_{z=0}^{p-1} \frac{z^r}{(1-\alpha z^2)^r} \\
&=  \sum_{z=1}^{p-1} (z^{-1}-\alpha z)^{p-1-r}  \\
& = \begin{cases}
-(-1)^{(p-1-r)/2} {p-1-r \choose (p-1-r)/2}\alpha^{(p-1-r)/2} &\text{if $r$ is even,} \\
0 &\text{if $r$ is odd,}
\end{cases}
\end{align*}
once again by first applying the binomial theorem and then taking the sum.

Thus, we have proved:

\begin{proposition}\label{prop-main}
For the representation 
\[\rho = {\rm Sym}^{2r} \overline{\mathbb F}_p^2 \otimes \det{^{-r}} \overline{\mathbb F}_p^2,\]
we have
\[st = \begin{cases}
(-1)^{\frac{p-1}{2}} {r \choose \frac{r}{2}} {p-1-r \choose \frac{p-1-r}{2}} &\text{if $r$ is even,} \\
0 &\text{if $r$ is odd.}
\end{cases}
\]
\end{proposition}

\begin{remark}\label{rmk-handy}
Note the symmetry in $r$ and $p-1-r$ in Proposition \ref{prop-main}. This comes quite handy in formulating the result in the general case in the next subsection.
\end{remark}

\subsection{The general case}\label{case-2}

Now we have $G={\rm PGL}_2(\mathbb F_q)$ with $q=p^f$. For an irreducible complex representation of $\pi$ of $G$, we have isolated an irreducible mod $p$ representation in Proposition \ref{prop-future-1} and Proposition \ref{prop-future-1}. This representation is of the form
\[ \left( \bigotimes_{i \in J} \left( {\rm Sym}^{2r_i} \overline{\mathbb F}_p^2 \otimes \det{^{-r_i}} \overline{\mathbb F}_p^2 \right) \circ {\rm Frob}^i \right) \bigotimes \left(  \bigotimes_{i \notin J} \left( {\rm Sym}^{2p-2-2r_i} \overline{\mathbb F}_p^2 \otimes \det{^{r_i}} \overline{\mathbb F}_p^2 \right) \circ {\rm Frob}^i \right)\] 
where $J$ depends on the $p$-adic expansion
\[r= \begin{cases}
r_0+r_1p+\dots+r_{f-1}p^{f-1} &\text{if $\pi={\rm Ps}(\chi^{-r},\chi^r)$,} \\
1+r_0+r_1p+\dots+r_{f-1}p^{f-1} &\text{if $\pi=\pi(\psi^{(q-1)r})$.}
\end{cases}\]
But the dependence on $J$ is not a problem for us to proceed because of the symmetry between $J$ and $J^\mathsf{c}$ in terms of $r_i$ for $i \in J$ and $p-1-r_i$ for $i \notin J$. This is the same symmetry in Proposition \ref{prop-main} as well (cf. Remark \ref{rmk-handy}). Note also that $r_i$ is even if and only $p-1-r_i$ is even. Thus we are led to analyse a representation of ${\rm PGL}_2(\mathbb F_q)$ of the form
\[ \rho_r = \bigotimes_{i=0}^{f-1} \left( {\rm Sym}^{2r_i} \overline{\mathbb F}_p^2 \otimes \det{^{-r_i}} \overline{\mathbb F}_p^2 \right) \circ {\rm Frob}^i. \] 
Note that $\rho_r$ has an $H$-fixed (resp. $K$-fixed) vector precisely when each factor in its tensor product has an $H$-fixed (resp. $K$-fixed) vector. By Proposition \ref{prop-main}, this is equivalent to saying that each $r_i$ is even. If this is the case, then we can take
\[v_H = \bigotimes_{i=0}^{f-1} x_i^{r_i}y_i^{r_i}\]
and
\[v_K = \bigotimes_{i=0}^{f-1} (\alpha^{p^i} x_i^2-y_i^2)^{r_i}.\]
With $X,Y,s,t$ defined for $\rho_r$ as earlier and by denoting the corresponding scalars for the $i$-th component by $s_i,t_i$, we see that
\begin{align*}
st &= \prod_{i=0}^{f-1} s_it_i \\
&= \begin{cases}
\prod_{i=0}^{f-1} (-1)^{\frac{p-1}{2}} {r_i \choose r_i/2} {p-1-r_i \choose (p-1-r_i)/2} &\text{if each $r_i$ is even,} \\
0 &\text{otherwise,}
\end{cases} \\
&= \begin{cases}
(-1)^{\frac{q-1}{2}} {r \choose r/2} {q-1-r \choose (q-1-r)/2} &\text{if each $r_i$ is even,} \\
0 &\text{otherwise,}
\end{cases}
\end{align*}
where in the last step we have made use of Lucas's theorem according to which
\[{m \choose n} \equiv \prod_{i=0}^f {m_i \choose n_i} \mod p\] 
if $m = m_0+m_1p+\dots+m_{f-1}p^{f-1}$ and $n=n_0+n_1p+\dots+n_{f-1}p^{f-1}$. 

\begin{remark}\label{rmk-nonzero}
Note that the answer in the case of each $r_i$ being even is non-zero.
\end{remark}

We have thus proved the following theorem.

\begin{theorem}\label{thm-above}
Let $\pi$ be an irreducible complex representation of $G={\rm PGL}_2(\mathbb F_q)$ with unique fixed vectors, up to multiplication by scalars, for its split torus $H$ and a non-split torus $K$. Thus, $\pi$ can be a principal series ${\rm Ps}(\chi^r,\chi^{-r})$, where $\chi:\mathbb F_q^\times \rightarrow \mathbb C^\times$ reduces mod $p$ to the identity character $x \mapsto x$, or the twisted Steinberg St$\otimes \eta$, where $\eta$ is the quadratic character of $G$, or a cuspidal representation $\pi(\psi^{(q-1)r})$, where $\psi:\mathbb F_{q^2}^\times \rightarrow \mathbb C^\times$ reduces mod $p$ to the identity character $x \mapsto x$. Write
\[r= \begin{cases}
r_0+r_1p+\dots+r_{f-1}p^{f-1} &\text{if $\pi={\rm Ps}(\chi^r,\chi^{-r})$,} \\
\frac{q-1}{2} &\text{if $\pi = {\rm St}\otimes \eta$,} \\
1+r_0+r_1p+\dots+r_{f-1}p^{f-1} &\text{if $\pi=\pi(\psi^{(q-1)r})$.}
\end{cases}\] Define
\[
\varepsilon_\pi = 
\begin{cases}
(-1)^r &\text{if $\pi={\rm Ps}(\chi^r,\chi^{-r})$,} \\
(-1)^{\frac{q-1}{2}} &\text{if $\pi={\rm St}\otimes \eta$,} \\
(-1)^{r-1} &\text{if $\pi = \pi(\psi^{(q-1)r})$.}
\end{cases}
\]
Let $v_H$ and $v_K$ denote respectively the $H$ and $K$ fixed vectors of unit norm in the standard $G$-invariant inner product. Then,
\begin{enumerate}
\item If $\varepsilon_\pi = -1$, then $\langle v_H,v_K \rangle_\pi = 0$.
\item Modulo $p$, we have
\[|\langle v_H,v_K \rangle_\pi|^2 =  \begin{cases}
(-1)^{\frac{q-1}{2}} {r \choose r/2} {q-1-r \choose (q-1-r)/2} &\text{if each $r_i$ is even,} \\
0 &\text{otherwise,}
\end{cases}\]
and in particular, $\langle v_H,v_K \rangle_\pi \neq 0$ if each $r_i$ is even.
\end{enumerate}    
\end{theorem}

\begin{proof}
Part (1) follows from Lemma \ref{lemma-sign} together with Corollary \ref{cor-sign}. Part (2) is proved in this section.
\end{proof}

\begin{remark}\label{rmk-modp}
Note that for any $q$, exactly as in the case of $q=p$, $s$ and $t$ are given by
\[s=(-1)^{r/2}{r \choose r/2}\alpha^{r/2},\]
and
\[t=-(-1)^{(q-1-r)/2}{q-1-r \choose (q-1-r)/2}\alpha^{(q-1-r)/2}.\]
It is the product $st$ that is relevant for us in Theorem \ref{thm-above}, and not individual $s$ or $t$. Note that though $s, t \in \mathbb F_q$, we have $st \in \mathbb F_p$. Thus, it is a consequence of our proof that we are able to take $|\langle v_H,v_K \rangle_\pi|^2$ modulo $p$ (and not just modulo a prime lying above $p$, which need not be a well-defined map as we work with a fixed identification $\overline{\mathbb Q}_p \simeq \mathbb C$). Also, observe that we cannot go modulo $p$ for $\langle v_H,v_K \rangle_\pi$ in general, as can be seen, for instance, from its expression in Proposition \ref{prop-ip} below.
\end{remark}

\section{Examples}\label{examples}

\subsection{}\label{example1}

We consider the three cuspidal representations of ${\rm PGL}_2(\mathbb F_7)$. The correlation coefficient for these three representations are $0,(2-\sqrt{2})/6,(2+\sqrt{2})/6$, as can be checked using Lemma \ref{lemma-elementary}. Let $\pi$ be such that the correlation coefficient is $(2-\sqrt{2})/6$. Now $7 \mathbb Z[\zeta_{49}]$ is a product of eight prime ideals in $\mathbb Z[\zeta_{49}]$, and for four of these, the corresponding mod $7$ value is $(2-3)/6 = 1$ and for the other four this value is $(2-4)/6 = 2$. Thus, the value in $\mathbb Q(\zeta_{49})_\mathfrak p$ depends on the choice of $\mathfrak p | 7$. 

\subsection{}\label{example2}

As mentioned in Section \ref{intro} (cf. Remark \ref{rmk-vatsal}), $r$ and the $r_i$'s in Theorem \ref{thm-above} depend on the choice of a prime $\mathfrak p$ of $\overline{\mathbb Q}$ lying above $p$,. In many cases, this observation on a choice of $\mathfrak p | p$ can be exploited to conclude the non-vanishing of the correlation coefficient $|\langle v_H,v_K \rangle_\pi|^2$. The simplest example to illustrate this principle can be constructed with a principal series representation of ${\rm PGL}_2(\mathbb F_{49})$. Let $\zeta_{48} = \exp(2\pi\imath/48)$. Let $\mathfrak p_1$ and $\mathfrak p_2$ be prime ideals in $\mathbb Z[\zeta_{48}]$ lying above $7$ corresponding respectively to the polynomial factors of the irreducible polynomial of $\zeta_{48}$ modulo $7$ given by $p_1(X)=X^2 + 6 X + 3$ and $p_2(X)= X^2 + 4 X + 5$. Let
\[\mathbb E_1 = \frac{\mathbb F_7[X]}{(X^2+6X+3)} ~\&~ \mathbb E_2 = \frac{\mathbb F_7[X]}{(X^2+4X+5)}.\]
Check that $x \mapsto x^5$ is an isomorphism from $\mathbb E_1^\times$ to $\mathbb E_2^\times$. Let $\chi:\mathbb E_1^\times \rightarrow \mathbb C^\times$ be given by $X \mapsto \zeta_{48}$.
Thus, if we denote by $\chi_\mathfrak p:\mathbb F_q^\times \rightarrow \mathbb C^\times$ the character that reduces modulo $\mathfrak p$ to the identity character of $\mathbb F_q^\times$, then,
$\chi = \chi_{\mathfrak p_1} = \chi_{\mathfrak p_5}^5$.
Let $\pi ={\rm Ps}(\chi^{10},\chi^{-10})={\rm Ps}(\chi_{\mathfrak p_1}^{10},\chi_{\mathfrak p_1}^{-10}) = {\rm Ps}(\chi_{\mathfrak p_2}^{2},\chi_{\mathfrak p_2}^{-2})$. Since $10 = 3 + 1 \cdot 7$, it follows that 
\begin{equation}\label{eqn-eg1}
|\langle v_H,v_K \rangle_\pi|^2 \equiv 0 \mod 7,
\end{equation}
under the identification of $\overline{\mathbb Q}_p$ with $\mathbb C$ via $\mathfrak p_1$. So at this point we cannot conclude anything about the non-vanishing of $|\langle v_H,v_K \rangle_\pi|^2$. However, with respect to the identification via $\mathfrak p_2$, since
$2 = 2 + 0 \cdot 7$,
we see that
\begin{align}\label{eqn-eg2}
|\langle v_H,v_K \rangle_\pi|^2 & \equiv (-1)^3{2 \choose 1}{0 \choose 0}(-1)^3{4 \choose 2}{6 \choose 3} \mod 7 \equiv 2 \mod 7.
\end{align}
Therefore, $|\langle v_H,v_K \rangle_\pi|^2 \neq0$. Indeed, it can be verified, by Proposition \ref{prop-ip}, that
\[|\langle v_H,v_K \rangle_\pi|^2 = \frac{\zeta_{48}^{10}- \zeta_{48}^6- \zeta_{48}^2+10}{150} = \frac{10-\sqrt{2}}{150}.\]
Observe that, modulo $p_1(X)$, we have $X^{10}-X^6-X^2+10=0$ in $\mathbb F_7[X]$, verifying (\ref{eqn-eg1}), and, modulo $p_2(X)$, we have $X^{10}-X^6-X^2+10=6$ in $\mathbb F_7[X]$, verifying (\ref{eqn-eg2}).    

\subsection{}\label{example3}

Thanks to Theorem \ref{thm2}, we know that $\langle v_H,v_K \rangle_\pi = 0$ if $\pi={\rm BC}(\tau)$ with $\varepsilon_\tau = -1$. If $q=p^f$ with $f$ even, then note that $\varepsilon_\pi = \varepsilon_\tau^f = 1$. We could ask whether there are examples of $\pi$ for which $\langle v_H,v_K \rangle_\pi = 0$ with $\varepsilon_\pi = 1$ but $\pi$ is not a base change.

The answer is yes. For instance, for $\pi = {\rm Ps}(\chi^{24},\chi^{-24})$ of of ${\rm PGL}_2(\mathbb F_{289})$, for a character $\chi$ of $\mathbb F_{288}^\times$ as in Section \ref{intro}, using Proposition \ref{prop-ip}, we check that $\langle v_H,v_K \rangle_\pi = 0$. Note that $\varepsilon_\pi = 1$ and that this representation is not a base change from ${\rm PGL}_2(\mathbb F_{17})$.

\subsection{}\label{example4}

It is natural to ask, in light of Theorem \ref{thm1} (2), the following question. If
\[|\langle v_H,v_K \rangle_\pi|^2 \equiv 0 \mod p,\]
for all the identifications of $\overline{\mathbb Q}_p$ with $\mathbb C$ via all the primes $\mathfrak p$ of $\mathbb Z[\zeta_k]$ lying above $p$, where $k = q-1$ (resp. $k= q^2-1$) when $\pi$ is a principal series (resp. cuspidal) representation of ${\rm PGL}_2(\mathbb F_q)$, then is it true that $\langle v_H,v_K \rangle_\pi = 0$?

The answer is no. For instance, for $\pi = {\rm Ps}(\chi^{38},\chi^{-38})$ of ${\rm PGL}_2(\mathbb F_{343})$, for a character $\chi$ of $\mathbb F_{342}^\times$ as in Section \ref{intro}, using Proposition \ref{prop-ip}, we check that
\[|\langle v_H,v_K \rangle_\pi|^2 = \frac{588}{342 \times 344} \equiv 0 \mod 7,\]
under any identification of $\overline{\mathbb Q}_p$ with $\mathbb C$ via $\mathfrak p | 7$ of $\mathbb Z[\zeta_{342}]$.

\section{A direct approach}\label{direct}

A direct approach to proving Theorem \ref{thm1} can be tried for a principal series representation ${\rm Ps}(\chi,\chi^{-1})$ of $G={\rm PGL}_2(\mathbb F_q)$ as it admits a model where the invariant vectors for $H$ and $K$ can be explicitly written down. We can compute $\langle v_H, v_K \rangle_\pi$ in terms of the inducing character $\chi$ but the value of the resulting character sum seems to be difficult to determine.

Let $\pi={\rm Ps}(\chi,\chi^{-1})$. The space of $\pi$ is given by
\[\left\{f: G \rightarrow \mathbb C \mid f\left( \left( \begin{array}{cc} a & b \\ 0 & 1 \end{array}\right) g\right) = \chi(a)f(g) \right\}. \] 
By the Bruhat decomposition, a basis for this space is given by
\begin{equation}
f(g) = \begin{cases} 
1 &\text{if $g=I$,} \\
0 &\text{otherwise,}
\end{cases}
\end{equation}
and, for $\lambda \in \mathbb F_q$,
\begin{equation}\label{eqn-lambda}
f_\lambda\left(  \left( \begin{array}{cc} 0 & 1 \\ 1 & 0 \end{array}\right)  \left( \begin{array}{cc} 1 & \mu \\ 0 & 1 \end{array}\right) \right) = \begin{cases} 
1 &\text{if $\mu=\lambda$,} \\
0 &\text{otherwise.}
\end{cases}
\end{equation}
Then it is a direct check to show that the unit norm vector
\begin{equation}\label{eqn-vh}
v_H = \frac{1}{|H|}\sum_{0 \neq \lambda \in \mathbb F_q} \chi^{-1}(\lambda)f_\lambda
\end{equation}
 is $H$-invariant and the unit norm vector
\begin{equation}\label{eqn-vk}
v_K = \frac{1}{|K|} \left(f+ \sum_{\lambda \in \mathbb F_q}\chi\left(\frac{1}{\alpha-\lambda^2} \right)f_\lambda \right)
\end{equation}
is $K$-invariant. Thus we get the following proposition.
\begin{proposition}\label{prop-ip}
For $v_H$ and $v_K$ as above, we have:
\[\langle v_H, v_K \rangle_\pi = \frac{1}{|H|}\frac{1}{|K|} \sum_{0 \neq \lambda \in \mathbb F_q}\chi(\alpha\lambda^{-1}-\lambda).\]
 \end{proposition}
 \begin{remark}
 Note that $|\langle v_H, v_K \rangle_\pi|$ depends only on the isomorphism class of $\pi$.
 \end{remark}
 \begin{remark}
 By changing $\lambda$ to $-\lambda$, it follows that $\langle v_H, v_K \rangle_\pi = 0$ if $\chi(-1)=-1$.
 \end{remark}
 
 Suppose $a \in \mathbb F_q^\times$ is such that 
 \[k_{\alpha a^2} = \left(\begin{array}{cc} 1 & \alpha a^2 \\ 1 & 1 \end{array} \right)\] also has order $q+1$ in $G$. Let $K_a$ denote the cyclic subgroup of $G$ generated by $k_{\alpha a^2}$ (cf. Remark \ref{rmk-k}). Define $v_{K_a}$ as above. 
 
 Making the change $\lambda \rightarrow \lambda a$ in Proposition \ref{prop-ip}, we get the following corollary.
 \begin{corollary}\label{coro-ip}
 We have
 \[\langle v_H, v_{K_a} \rangle_\pi = \chi(a) \langle v_H, v_K \rangle_\pi. \]
 \end{corollary}
 
Now we consider an extension of finite fields of degree $f$. Let $\sigma$ be the Frobenius morphism of $\mathbb F_{q^f}$ given by $\sigma(x)=x^q$ and thus $\sigma$ is of order $f$. Let $\sigma$ act on $g \in G$ component-wise.
Let 
\[K^\sigma = \{k^\sigma \mid k \in K\}.\]   
Note that $K^\sigma = K_a$ for $a=\alpha^{\frac{q-1}{2}}$. In particular, we have (cf. Corollary \ref{coro-ip}):
\begin{equation}\label{eqn-ip}
\langle v_H, v_{K^\sigma} \rangle_\pi = \chi(\alpha^{\frac{q-1}{2}}) \langle v_H, v_K \rangle_\pi.
\end{equation}

\section{Shintani base change}\label{bc}

In this section we discuss the effect of Shintani base change \cite{shi76,kaw77} on the behaviour of $\langle v_H, v_K \rangle$. As in the end of the last section, we consider an extension of finite fields of degree $f$ and $\sigma$ denotes the Frobenius morphism of $\mathbb F_{q^f}$ given by $\sigma(x)=x^q$. It will be convenient to denote $\mathbb F = \mathbb F_{q}$ and $\mathbb E = \mathbb F_{q^f}$ which we do. The main result in this section is that the invariant vectors for $H = H_\mathbb E$ and $K = K_\mathbb E$ are orthogonal for a representation $\pi$ of ${\rm PGL}_2(\mathbb E)$ if $\pi$ is obtained as the base change of a representation $\tau$ of ${\rm PGL}_2(\mathbb F)$ with $\varepsilon_\tau = -1$.

We first recall the Shintani base change \cite{shi76,kaw77}. If $\pi$ is a representation of $G = {\rm PGL}_2(\mathbb E)$, then its $\sigma$-conjugate is defined by 
\[\pi^\sigma(g)=\pi(g^\sigma).\] 
Assume that $\pi$ is a base change of an irreducible representation of ${\rm PGL}_2(\mathbb F)$. Thus, we have $\pi \simeq \pi^\sigma$. There are two cases to consider.
\begin{enumerate}
\item $\chi=\chi^\sigma$, in which case $\pi$ is the base change of a principal series representation of ${\rm PGL}_2(\mathbb F)$.
\item $\chi^{-1}=\chi^\sigma$, in which case $\pi$ is the base change of a cuspidal representation of ${\rm PGL}_2(\mathbb F)$. Note that in this case, we have $\chi=\chi^{\sigma^2}$, and thus $f$ is even and $\chi=\chi^\prime \circ {\rm Nm}_{\mathbb F_{q^f}/\mathbb F_{q^2}}$, for a character $\chi^\prime$ of $\mathbb F_{q^2}^\times$. And $\pi$ is the base change of the cuspidal representation of ${\rm PGL}_2(\mathbb F)$ associated to $\chi^\prime$.
\end{enumerate}

Since $\pi$ is $\sigma$-invariant, there exists an operator $T$ on the space of $\pi$, which is unique up to scalars, such that
\[\pi(g)T=T\pi(g^\sigma)\] for all $g \in G$. Now if $\tau$ is the irreducible representation of ${\rm PGL}_2(\mathbb F)$ that base changes to $\pi$, then we have 
\[{\rm Trace}(\pi(g)T) = \chi_\tau({\rm N}(g)),\] 
where 
\[{\rm N}(g) = gg^\sigma \dots g^{\sigma^{f-1}}.\]
Note that though N$(g)$ need not be in ${\rm PGL}_2(\mathbb F)$, it is conjugate to an element of ${\rm PGL}_2(\mathbb F)$, by Lang's theorem, and thus $\chi_\tau({\rm N}(g))$ is well-defined. 

We do not make direct use of the Shintani character identity and we need it only to rephrase our results in terms of base change.

In the next proposition, we explicitly state the Shintani operator on $\pi={\rm Ps}(\chi,\chi^{-1})$ which is $\sigma$-invariant. 
\begin{proposition}\label{prop-shintani}
We have:
\begin{enumerate}
\item\label{shintani1} When $\chi=\chi^\sigma$, we may take the Shintani operator $T$ on $\pi$ to be 
\begin{align*}
Tf &= f, \\
Tf_\lambda &= f_{\lambda^\sigma}.
\end{align*}
\item When $\chi^{-1}=\chi^\sigma$, write $f=2n$. We may take the Shintani operator $T$ on $\pi$ to be 
\begin{align*}
Tf & = \frac{(-1)^{n-1}}{q^n} \sum_{\mu \in \mathbb E} f_\mu, \\
Tf_\lambda & = \frac{(-1)^{n-1}}{q^n} \left(f+\sum_{\lambda \neq \mu \in \mathbb E} \chi(-1)\chi(\lambda-\mu)^2f_{\mu^\sigma}\right).
\end{align*}
\end{enumerate}
\end{proposition}
\begin{proof}
A direct check shows
\[\left(\begin{array}{cc} a & 0 \\ 0 & 1 \end{array} \right) f  = \chi(a) f,  
\left(\begin{array}{cc} a & 0 \\ 0 & 1 \end{array} \right)f_\lambda  = \chi^{-1}(a)f_{a\lambda},\]
\[\left(\begin{array}{cc} 1 & b \\ 0 & 1 \end{array} \right) f  = f,
\left(\begin{array}{cc} 1 & b \\ 0 & 1 \end{array} \right) f_\lambda = f_{\lambda-b},\]
\[\left(\begin{array}{cc} 0 & 1 \\ 1 & 0 \end{array} \right) f = f_0,
\left(\begin{array}{cc} 0 & 1 \\ 1 & 0 \end{array} \right) f_0 = f,
\left(\begin{array}{cc} 0 & 1 \\ 1 & 0 \end{array} \right) f_\lambda = \chi(-\lambda^2) f_{\lambda^{-1}}, \]
where the very last identity is for $\lambda \neq 0$.
It is a direct verification, using the Bruhat decomposition, that we have
\[\pi(g)T=T\pi(g^\sigma)\] for all $g \in G$, which is what we had to show.
\end{proof}

In Proposition \ref{prop-shintani}, we have normalized the Shintani operator $T$ so that it matches with the standard normalization of the Shintani operator on the Whittaker model $\mathcal W(\pi,\psi)$ of $\pi$ with respect to an additive character $\psi$ of $\mathbb E$. Here, $\psi$ is chosen to be $\sigma$-invariant; i.e., $\psi=\psi^\sigma$. We need to prove this non-obvious claim. 

Recall that the standard Shintani operator on $\mathcal W(\pi,\psi)$ is given by
\[TW=W^\sigma\]
where $W^\sigma(g)=W(g^\sigma)$. Note that $W^\sigma$ naturally belongs to $\mathcal W(\pi^\sigma,\psi^\sigma)$ which equals $\mathcal W(\pi,\psi)$ by our assumptions on $\pi$ and $\psi$ and because of the uniqueness of the Whittaker model.     

As a first step towards proving the claim on the matching of the Shintani operators defined respectively on the space of $\pi$ and on $\mathcal W(\pi,\psi)$ under the isomorphism
\[\pi \simeq \mathcal W(\pi,\psi) \subset {\rm Ind}_N^G \psi,\]   
where $N$ is the unipotent radical of the Borel subgroup of $G={\rm PGL}_2(\mathbb F)$, we show that the normalizing constant is correct up to a sign.

To this end, recall that for $f \in \pi$, the corresponding Whittaker function in $\mathcal W(\pi,\psi)$  is given by
\[W_f(g) = \sum_{\mu \in \mathbb F} \psi^{-1}(\mu) f\left(\left( \begin{array}{cc} 0 & 1 \\ 1 & 0 \end{array} \right) \left( \begin{array}{cc} 1 & \mu \\ 0 & 1 \end{array} \right) g \right).\] Also, note that for the Whittaker function
\[W = \frac{1}{|N|} \sum_{\lambda \in \mathbb F} \psi(\lambda)W_{f_\lambda},\]
where the $f_\lambda$'s are as in (\ref{eqn-lambda}), we have
\[W(1)=1.\]

Now let $S$ denote the intertwining operator $T$ on $\pi$ given in Proposition \ref{prop-shintani} (2) except for the normalizing constant $(-1)^{n-1}/q^n$. We need to show that 
\[\frac{1}{|N|}\sum_{\lambda \in \mathbb F} \psi(\lambda)W_{Sf_\lambda}(1) = \pm q^n,\] so that $T$ given in Proposition \ref{prop-shintani} (2) is indeed the normalized Shintani operator up to a sign.   

To prove this identity, note that
\begin{align*}
W_{Sf_\lambda}(1) &= \sum_{\mu \in \mathbb F} \psi^{-1}(\mu) Sf_\lambda \left(\left( \begin{array}{cc} 0 & 1 \\ 1 & 0 \end{array} \right) \left( \begin{array}{cc} 1 & \mu \\ 0 & 1 \end{array} \right) \right) \\
&=  \sum_{\mu \in \mathbb F} \psi^{-1}(\mu)  \left[f+\sum_{x \neq \lambda} \chi(-1)\chi(x-\lambda)^2 f_{x^\sigma} \right] \left(\left( \begin{array}{cc} 0 & 1 \\ 1 & 0 \end{array} \right) \left( \begin{array}{cc} 1 & \mu \\ 0 & 1 \end{array} \right) \right) \\
&= \sum_{x \neq \lambda} \psi^{-1}(x)\chi(x-\lambda)^2,
\end{align*}
since $\psi=\psi^\sigma$ and $\chi(-1)=1$. Therefore,
\begin{align*}
\frac{1}{|N|}\sum_{\lambda \in \mathbb F} \psi(\lambda)W_{Sf_\lambda}(1)  &= \frac{1}{|N|}\sum_{\lambda \in \mathbb F} \sum_{x \neq \lambda} \psi(\lambda-x)\chi(x-\lambda)^2 \\
&= \frac{1}{|N|}\sum_{\lambda \in \mathbb F} \left( \sum_{t \neq 0} \psi(t)\chi(t)^2\right) \\
&= G(\chi^2,\psi),
\end{align*}
the Gauss sum of $\chi^2$ with respect to $\psi$.

From the well-known properties of the Gauss sum, as $\chi^{-1}=\chi^\sigma$ and $\psi=\psi^\sigma$, we get
\[G(\chi^2,\psi)^2 = G(\chi^2,\psi)G(\chi^{-2\sigma},\psi) =  G(\chi^2,\psi)G(\chi^{-2},\psi) = \chi^2(-1)q^{2n}=q^{2n},\]
and it follows that
\[G(\chi^2,\psi) = \pm q^n. \] 

\begin{remark}\label{rmk-bessel}
Indeed, we claim that the correct sign is $(-1)^{n-1}$. This will be a consequence of our analysis of the action of $T$ on the $H$-fixed vector.
\end{remark}

In the next lemma, we write down an explicit $H$-fixed vector on the Whittaker model $\mathcal W(\pi,\psi)$ in terms of the Bessel function \cite{mao01}. The normalized Bessel function is given by 
\[B_{\pi,\psi}(g) = \frac{1}{|N|} \sum_{n \in N} \psi^{-1}(n)\chi_\pi(gn),\]
where $\chi_\pi$ is the character of $\pi$. Note that $B_{\pi,\psi}(1)=1$. Recall that the normalized Bessel function is the unique Whittaker function which is bi-$\psi$-invariant. It follows that, when $\pi \simeq \pi^\sigma$, 
\begin{equation}\label{eqn-bessel}
TB_{\pi,\psi} = B_{\pi,\psi}^\sigma = B_{\pi,\psi}.
\end{equation}

\begin{lemma}\label{lemma-bessel}
The unique non-trivial $H$-fixed vector, up to a scalar multiple, in $\mathcal W(\pi,\psi)$ is given by 
\[W_H = \sum_{h \in H}\pi(h)B_{\pi,\psi}.\]
\end{lemma}

\begin{proof}
As it is obviously $H$-invariant, we only need to check that it is indeed non-trivial. Note that
\[W_H(1)= \sum_{h \in H} B_{\pi,\psi}(h) = B_{\pi,\psi}(1)=1, \]
as it can be seen that $B_{\pi,\psi}(h)=0$ for $h \neq 1$ (an instance of Gelfand's lemma \cite[Proposition 4.9]{gel70}).
\end{proof}

\begin{corollary}\label{coro-bessel}
The standard Shintani operator $T$ on $\mathcal W(\pi,\psi)$ fixes any $H$-fixed vector.
\end{corollary}

\begin{proof}
We have
\[TW_H = \sum_{h \in H} T\pi(h)B_{\pi,\psi} =  \sum_{h \in H} \pi(h^\sigma)TB_{\pi,\psi} =  \sum_{h \in H} \pi(h^\sigma)B_{\pi,\psi} =  \sum_{h \in H} \pi(h)B_{\pi,\psi} = W_H,\]
by (\ref{eqn-bessel}).
\end{proof}

Next we look at the effect of the Shintani operator $T$, defined on the space of $\pi$ (cf. Proposition \ref{prop-shintani}), on the $H$ and $K$-invariant vectors.

\begin{proposition}\label{prop-effect}
Let $\pi = {\rm Ps}(\chi,\chi^{-1})$ be a $\sigma$-invariant representation of ${\rm PGL}_2(\mathbb E)$ so that $\pi={\rm BC}(\tau)$, where $\tau$ is an irreducible representation of ${\rm PGL}_2(\mathbb F)$. 
\begin{enumerate}
\item\label{ip-1} If $\chi=\chi^\sigma$, in which case $\tau$ is a principal series representation, we have
\[Tv_H = v_H ~\&~ Tv_K = v_{K^\sigma}.\]
\item\label{ip-2} If $\chi^{-1}=\chi^\sigma$, in which case $\tau$ is a cuspidal representation, we have
\[Tv_H = v_H ~\&~ Tv_K = -\chi(\alpha^{-1}) v_{K^\sigma}.\]
\end{enumerate}
\end{proposition}

\begin{proof}
See the definitions of $v_H$ and $v_K$ given by (\ref{eqn-vh}) and (\ref{eqn-vk}). Similarly, $v_K^\sigma$ is defined by (\ref{eqn-vk}) except that $\alpha$ is replaced by $\alpha^\sigma$ there. 

Now the proof of (\ref{ip-1}) is immediate by Proposition \ref{prop-shintani} (\ref{shintani1}). Indeed, apply $T$, make the change $\lambda \rightarrow \lambda^\sigma$ and use the fact that we have $\chi=\chi^\sigma$.

Next we take up (\ref{ip-2}). Recall that in this case we have $\mathbb E = \mathbb F_{q^f}$ where $f=2n$. Note that
\begin{align*}
Tv_H &= \frac{(-1)^{n-1}}{q^n}\frac{1}{|H|}  \sum_{\lambda \neq 0} \chi^{-1}(\lambda) \left( f+ \sum_{\mu \neq \lambda} \chi(-1)\chi(\lambda-\mu)^2f_{\mu^\sigma}\right) \\
&= \frac{(-1)^{n-1}}{q^n}\frac{1}{|H|}   \sum_{\lambda \neq 0} \sum_{\mu \neq \lambda} \chi(-1)\chi(\lambda-2\mu+\mu^2 \lambda^{-1})f_{\mu^\sigma} \\
&=  \frac{(-1)^{n-1}}{q^n}\frac{1}{|H|}   \sum_{\lambda \neq 0,1} \sum_{\mu \neq 0} \chi(-1)\chi(\mu)\chi(\lambda-2+\lambda^{-1})f_{\mu^\sigma} \\
&=  \frac{(-1)^{n-1}}{q^n}\frac{1}{|H|}   \sum_{\lambda \neq 0,1} \sum_{\mu \neq 0} \chi(-1)\chi^{-1}(\mu)\chi(\lambda-2+\lambda^{-1})f_{\mu} 
\end{align*}
by making the change $\lambda \rightarrow \lambda\mu$ and then $\mu \rightarrow \mu^\sigma$ and observing that $\chi^{-1}=\chi^\sigma$. Thus, $Tv_H=v_H$ is equivalent to showing that
\[\sum_{0, 1 \neq \lambda \in \mathbb E}\chi(\lambda-2+\lambda^{-1}) = (-1)^{n-1}q^n,\]
since $\chi(-1)=1$. 

But we already know that the operator $T$ on the space of $\pi$ matches with the operator $T$ on $\mathcal W(\pi,\psi)$, under the identification of $\pi$ with $\mathcal W(\pi,\psi)$, up to a sign (cf. Remark \ref{rmk-bessel}). Thus, it follows from Corollary \ref{coro-bessel}, and the computation above, that
\begin{equation}\label{eqn-sign}
\sum_{0, 1 \neq \lambda \in \mathbb E}\chi(\lambda-2+\lambda^{-1}) = \pm q^n,
\end{equation}
where we now additionally need to show that the sign is indeed $(-1)^{n-1}$. The analysis above shows that the sign is independent of the character $\chi$. In order to fix this sign, we sum the identities in (\ref{eqn-sign}) over the $(q+1)$ many characters $\chi$ of $\mathbb E^\times$ such that $\chi^{-1}=\chi^\sigma$. Note that, by orthogonality relations,
\[\sum_{\chi^{-1}=\chi^\sigma} \sum_{0, 1 \neq \lambda \in \mathbb E}\chi(\lambda-2+\lambda^{-1}) =
 \sum_{0, 1 \neq \lambda \in \mathbb E} \left( \sum_{\chi^{-1}=\chi^\sigma} \chi(\lambda-2+\lambda^{-1})\right)  \]
is a multiple of $(q+1)$. On the other hand the double sum on the left hand side is
\[(q^{2n}-2)+(-1)+(q+1-2) (\cdot) q^n = q^{2n}-3+(q-1)(\cdot)q^n\]  
by summing over the trivial character, the quadratic character, and the rest $(q-1)$ many characters. Now this expression is a multiple of $(q+1)$ implies that
\[(\cdot) = (-1)^{n-1}.\]  

Now we show that $Tv_K = -\chi^{-1}(\alpha)v_{K^\sigma}$. To this end, note that
\begin{align*}
Tv_K &= \frac{1}{|K|} \left[Tf+\sum_\lambda \chi\left(\frac{1}{\alpha-\lambda^2} \right)Tf_\lambda \right] \\
&=  \frac{(-1)^{n-1}}{q^n}\frac{1}{|K|}  \left[\sum_\mu f_\mu +\sum_\lambda \chi \left(\frac{1}{\alpha-\lambda^2} \right) \left( f+\sum_{\mu \neq \lambda} \chi(-1)\chi(\lambda-\mu)^2f_{\mu^\sigma} \right) \right] \\
&=  \frac{(-1)^{n-1}}{q^n}\frac{1}{|K|} \left[ \left( \sum_\lambda \chi \left(\frac{1}{\alpha-\lambda^2}\right)\right) f+  \sum_\mu \left(1+\sum_{\lambda \neq \mu} \chi(-1) \chi \left(\frac{(\lambda-\mu)^2}{\alpha-\lambda^2} \right) \right) f_{\mu^\sigma} \right].
\end{align*}
Now observe that $Tv_K$ and $v_{K^\sigma}$ differ by a scalar since
\[\pi(k^\sigma)Tv_K = T\pi(k)v_K = Tv_K\]
which implies that $Tv_K$ is $K^\sigma$-invariant. We write
\[Tv_K = -\chi^{-1}(\alpha)t \cdot v_{K^\sigma}.\]
Therefore we get 
\begin{equation}\label{eqn-9}
\sum_\lambda \chi\left(\frac{1}{\alpha-\lambda^2} \right) = -(-1)^{n-1}q^n\chi^{-1}(\alpha)t,
\end{equation}
and
\begin{equation}\label{eqn-10}
1+\sum_{\lambda \neq \mu} \chi(-1) \chi \left(\frac{(\lambda-\mu)^2}{\alpha-\lambda^2} \right)   =  -(-1)^{n-1}q^n \chi^{-1}(\alpha)t ~ \chi\left(\frac{1}{\alpha^\sigma -\mu^{2\sigma}} \right).
\end{equation}
Note that (\ref{eqn-9}) is equivalent to 
\begin{equation}\label{eqn-11}
\sum_{i=1}^{q^{2n}-1} \chi^{-1}(1-\alpha^{2i-1}) = -1+(-1)^n q^n t.
\end{equation}
We sum (\ref{eqn-10}) over $\mu \neq 0$ and get
\begin{equation}\label{eqn-12}
q^{2n}-1-\sum_{\lambda \neq 0} \chi\left(\frac{\lambda^2}{\alpha-\lambda^2} \right) = (-1)^nq^n\chi^{-1}(\alpha)t \sum_{\mu \neq 0}\chi\left(\frac{1}{\alpha^\sigma-\mu^{2\sigma}}\right),
\end{equation}
since $\chi(-1)=1$. Note that
\[\sum_{\lambda \neq 0} \chi\left(\frac{\lambda^2}{\alpha-\lambda^2} \right) = \sum_{i=1}^{q^{2n}-1} \chi^{-1}(1-\alpha^{2i-1}) \]
and
\[\sum_{\mu \neq 0}\chi\left(\frac{1}{\alpha^\sigma-\mu^{2\sigma}}\right) = \chi(\alpha) \sum_{i=1}^{q^{2n}-1} \chi^{-1}(1-\alpha^{2i-1}), \]
since $\chi^{-1}=\chi^\sigma$. Substituting in (\ref{eqn-12}), we see that $t=\pm 1$. Summing (\ref{eqn-11}) over the $(q+1)$ many characters $\chi$ such that $\chi^{-1}=\chi^\sigma$, and applying orthogonality relations, we conclude that $t=1$ for each such $\chi$ and this is what we needed to prove.
\end{proof}

The following two lemmas on character sums over finite fields are corollaries to the proof of Proposition \ref{prop-effect} (\ref{ip-2}). Lemma \ref{lemma-sum} is related to $Tv_H=v_H$ and Lemma \ref{lemma-sum-2} is related to $Tv_K=-\chi^{-1}(\alpha)v_{K^\sigma}$.

\begin{lemma}\label{lemma-sum}
Let $\chi$ be a character of $\mathbb F_{q^{2n}}^\times$ such that $\chi^{-1}=\chi^\sigma$ where $\sigma(x)=x^q$. Then we have:
\begin{equation*}\label{eqn-identity}
\sum_{0, 1 \neq \lambda \in \mathbb F_{q^{2n}}}\chi(\lambda-2+\lambda^{-1}) = (-1)^{n-1}q^n.
\end{equation*}
\end{lemma}

\begin{lemma}\label{lemma-sum-2}
Let $\chi$ be a character of $\mathbb F_{q^{2n}}^\times$ such that $\chi^{-1}=\chi^\sigma$ where $\sigma(x)=x^q$. Let $\alpha$ be a generator of $\mathbb F_{q^{2n}}^\times$. Then we have:
\[\sum_{i=1}^{q^{2n}-1} \chi(1-\alpha^{2i-1}) = -1+(-1)^n q^n.\]
\end{lemma}

Now we state the main result of this section.

\begin{theorem}\label{thm-bc}
Let $\mathbb E/\mathbb F$ be a finite extension of finite fields. Let $\pi={\rm Ps}(\chi,\chi^{-1})$ be an irreducible principal series representation of ${\rm PGL}_2(\mathbb E)$ which is the base change of an irreducible representation $\tau$ of ${\rm PGL}_2(\mathbb F)$. Let $H$ (resp. $K$) be a split (resp. non-split) torus of ${\rm PGL}_2(E)$. Let $v_H$ (resp. $v_K$) be the standard $H$-invariant (resp. $K$-invariant) unit vector in $\pi$ (cf. \S \ref{direct}). Then, we have
\[\langle v_H, v_K \rangle_\pi = \varepsilon_\tau \langle v_H, v_K \rangle_\pi, \]
where $\varepsilon_\tau$ is the sign associated to the representation $\tau$ (cf. Proposition \ref{prop-agree}).
In particular, $v_H$ and $v_K$ are orthogonal if $\varepsilon_\tau=-1$.
\end{theorem}

\begin{proof}
Let $T$ be the normalized Shintani operator on the space of $\pi$ defined in Proposition \ref{prop-shintani}. It is easy to verify that $T$ is self-adjoint. By (\ref{eqn-ip}), we have
\[\langle v_H, v_{K^\sigma} \rangle_\pi = \chi(\alpha^{\frac{q-1}{2}}) \langle v_H, v_K \rangle_\pi.\] 

By Proposition \ref{prop-effect}, by making use of the self-adjointness of $T$, we conclude that 
\[\langle v_H, v_K \rangle_\pi = \chi(\alpha^{\frac{q-1}{2}}) \langle v_H, v_K \rangle_\pi\]
when $\chi=\chi^\sigma$ so that $\chi=\chi_{_0}\circ {\rm Nm}_{\mathbb E/\mathbb F}$ and $\tau = {\rm Ps}(\chi_{_0},\chi_{_0}^{-1})$, for a character $\chi_{_0}$ of $\mathbb F^\times$, and
\[\langle v_H, v_K \rangle_\pi = -\chi(\alpha^{\frac{q+1}{2}}) \langle v_H, v_K \rangle_\pi\] 
when $\chi^{-1}=\chi^\sigma$ so that $\chi=\chi^\prime \circ {\rm Nm}_{\mathbb E/\mathbb F^\prime}$, where $\mathbb F^\prime$ is the quadratic extension of $\mathbb F$ and $\tau=r(\chi^\prime)$ is the cuspidal representation of ${\rm PGL}_2(\mathbb F)$ associated to $\chi^\prime$.

Let $\mathbb E/\mathbb F$ be of degree $f$. When $\tau={\rm Ps}(\chi_{_0},\chi_{_0}^{-1})$, note that
\[\chi(\alpha^{\frac{q-1}{2}}) = \chi_{_0}(\alpha^{\frac{q-1}{2}}\alpha^{\frac{q^2-q}{2}} \dots \alpha^{\frac{q^f-q^{f-1}}{2}}) = \chi_{_0}(-1) = \varepsilon_\tau.\]
When $\tau=r(\chi^\prime)$, we know that $f=2n$. Now note that
\[-\chi(\alpha^{\frac{q+1}{2}})= -\chi^\prime(\alpha^{\frac{q+1}{2}} \alpha^{\frac{q^3+q^2}{2}} \dots \alpha^{\frac{q^{f-1}+q^{f-2}}{2}}) = \varepsilon_\tau\] as 
\[\left(\alpha^{\frac{q^{f-1}+q^{f-2}+\dots+q+1}{2}}\right)^2=\alpha^{\frac{q^f-1}{q-1}} \in \mathbb F^\times \smallsetminus \mathbb F^{\times 2}.\]
\end{proof}

\begin{remark}
A similar statement as in Theorem \ref{thm-bc} holds true for the twisted Steinberg representation ${\rm St}\otimes \eta$ as well, where $\eta$ is the quadratic character of $\mathbb E^\times$. The above analysis can be done on the reducible principal series ${\rm Ps}(\chi,\chi)$ to derive this.
\end{remark}

\section{$L$-values and $\epsilon$-factors}\label{l-values}

The study of global toric periods is connected to an $L$-value \cite{wal85,wal91} and the corresponding local case involves a certain $\epsilon$-factor \cite{tun83,sai93,pra96}. In this section, we give an analogy in which we compare the vanishing or non-vanishing of $\langle v_H, v_K \rangle_\pi$ with that of an $L$-value. Just as the sign of the global root number gives some but not all the information about the vanishing or non-vanishing of the $L$-value, we see that the sign that we introduced holds analogous information about the vanishing or non-vanishing of $\langle v_H, v_K \rangle_\pi$. A negative sign is sufficient for vanishing, however it does not capture all the vanishing, and moreover the vanishing can be for subtle reasons!

The role that the sign $\varepsilon_\pi$ in Proposition \ref{prop-agree} plays in the vanishing of $\langle v_H, v_K \rangle_\pi$ is akin to the role that the global root number plays in the vanishing of an automorphic $L$-function. Indeed, for a cuspidal automorphic representation $\pi$ of ${\rm PGL}_2(\mathbb A_E)$ which is a base change of a cuspidal automorphic representation $\tau$ of of ${\rm PGL}_2(\mathbb A_F)$, where $E/F$ is a quadratic extension of number fields, we have the factorizations 
\[L(s,\pi) = L(s,\tau)L(s,\tau\otimes \omega)\]
and
\[ \epsilon(s,\pi) =  \epsilon(s,\tau)  \epsilon(s,\tau \otimes \omega)\]
where $\omega$ is the quadratic character associated to $E/F$. We also have the functional equation for any cuspidal representation $\rho$ given by
\[L(s,\rho) = \epsilon(s,\rho)L(1-s,\rho).\]

Thanks to the functional equation we conclude that $L(1/2,\pi)=0$ if $\epsilon(1/2,\pi)=-1$. This is similar in analogy to Part (1) of Theorem \ref{thm1}. But $L(1/2,\pi)=0$ can happen with $\epsilon(1/2,\pi)=1$ too. In particular, this situation arises in the context of base change as $\epsilon(1/2,\pi)=1$ with $\epsilon(1/2,\tau)=\epsilon(1/2,\tau \otimes \omega)=-1$. In this case, $L(1/2,\tau)=0$ and this forces $L(1/2,\pi)=0$. This is analogous to Theorem \ref{thm2} where the negative sign of the base changing representation forces the invariant vectors in the base changed representation to be orthogonal though the sign of the base changed representation can be positive. 

Finally, to prolong the analogy further, we note the crucial role played by mod $p$ representation theory in our proofs, reminiscent of questions regarding the vanishing and non-vanishing of automorphic $L$-values modulo $p$. 
 
\section*{Acknowledgements} 

The first named author would like to thank Dipendra Prasad for suggesting this problem and for many valuable conversations which were of great help in this work. The authors would also like to thank him for many helpful suggestions on this manuscript. The authors would like to thank Vinayak Vatsal for several fruitful suggestions and in particular for stressing the point in Remark \ref{rmk-vatsal} which has helped in improving the efficacy of Theorem \ref{thm1} (2), compared to the statement in an earlier version of the manuscript. They also thank Ravi Raghunathan for his helpful comments on this paper.


\end{document}